\documentclass[12pt]{amsart}
\usepackage{amssymb, amsmath,amsthm}
\newtheorem{thm}{Theorem}[section]
\newtheorem{prop}[thm]{Proposition}
\newtheorem{conj}[thm]{Conjecture}
\newtheorem{cor}[thm]{Corollary}
\newtheorem{lem}[thm]{Lemma}
\theoremstyle{definition}
\newtheorem{rem}[thm]{Remark}
\newtheorem{def1}[thm]{Definition}

\newcommand{\ra}{\rightarrow}
\newcommand{\bk}{\backslash}
\newcommand{\mc}{\mathcal}
\newcommand{\mb}{\mathbb}

\newcommand{\rad}{\text{rad}}

\newcommand{\lla}{\left\langle}
\newcommand{\e}{\epsilon}

\newcommand{\rra}{\right\rangle}
\newcommand{\mbf}{\boldsymbol}
\renewcommand{\bar}{\overline}

\begin{document}
\title[Effective Asymptotics for Multilinear Averages]{Effective Asymptotic Formulae for Multilinear Averages of Multiplicative Functions}
\author
{Oleksiy Klurman}
\address{D\'{e}partement de Math\'{e}matiques et de Statistique\\ Universit\'{e} de Montr\'{e}al\\
Montr\'{e}al, Qu\'{e}bec, Canada}
\email{lklurman@gmail.com}
\author{Alexander P. Mangerel}
\address{Department of Mathematics\\ University of Toronto\\
Toronto, Ontario, Canada}
\email{sacha.mangerel@mail.utoronto.ca}
\maketitle
\begin{abstract}
Let $f_1,\ldots,f_k : \mb{N} \ra \mb{C}$ be multiplicative functions taking values in the closed unit disc. Using an analytic approach in the spirit of Hal\'{a}sz' mean value theorem, we compute multidimensional averages $$x^{-l} \sum_{\mbf{n} \in [x]^l} \prod_{1 \leq j \leq k} f_j(L_j(\mbf{n}))$$
as $x \ra \infty$, where $[x] := [1,x]$ and $L_1,\ldots, L_k$ are affine linear forms that satisfy some natural conditions. Our approach gives a new proof of a result of Frantzikinakis and Host that is distinct from theirs, with \emph{explicit} main and error terms. \\
As an application of our formulae, we establish a \emph{local-to-global} principle for Gowers norms of multiplicative functions. We also compute the asymptotic densities of the sets of integers $n$ such that a given multiplicative function $f: \mb{N} \ra \{-1, 1\}$ yields a fixed sign pattern of length 3 or 4 on almost all 3- and 4-term arithmetic progressions, respectively, with first term $n$.
\end{abstract}
\section{Introduction}
\subsection{Main Theorems}
For $k,l \geq 2$, let $\mbf{L} := (\mbf{L}_1,\ldots,\mbf{L}_k)$ be a vector of $k$ (affine) linear forms $L_j : \mb{R}^l \ra \mb{R}$ with non-negative integer coefficients, i.e.,
\begin{equation*}
L_j(\mbf{n}) = \alpha_{j,0} + \sum_{1 \leq r \leq l} \alpha_{j,r} n_r,
\end{equation*}
where $(\alpha_{j,r})_{0 \leq r \leq l} \in \mb{N}_0^{l+1}$. We will call such a vector an \emph{integral system}. Assume moreover that $(\alpha_{j,1},\ldots,\alpha_{j,l}) = 1$, for each $j$, and that the forms are pairwise linearly independent. We will say that a system of forms that satisfies these properties is \emph{primitive}. We will concern ourselves throughout this paper with \emph{primitive integral systems} of affine linear forms. We remark that this primitivity assumption is merely technical and can be removed with more effort.\\
Let $\mb{U}$ denote the closed unit disc. We say that a function $f : \mb{N} \ra \mb{C}$ is \emph{1-bounded} if $f(n) \in \mb{U}$ for all $n$. For a vector $\mbf{f} := (f_1,\ldots,f_k)$ of 1-bounded multiplicative functions, a vector $\mbf{x} := (x_1,\ldots,x_l) \in (0,\infty)^l$ and a system of primitive integral affine linear forms $\mbf{L}$, put
\begin{equation*}
M(\mbf{x}; \mbf{f},\mbf{L}) := \langle \mbf{x}\rangle^{-1}  \sum_{\mbf{n} \in \mc{B}(\mbf{x})} \prod_{1 \leq j \leq k} f_j(L_j(\mbf{n})),
\end{equation*}
where $\mc{B}(\mbf{x})$ denotes the box $\prod_{1 \leq j \leq k} (0,x_j]$, and $\langle \mbf{x}\rangle = x_1 \cdots x_l$ is its volume.
When $\mbf{x} = (x,\ldots,x)$ for some $x \geq 1$ then we will write $M(x;\mbf{f},\mbf{L})$ instead. \\
The main purpose of the present paper is to establish an asymptotic formula for $M(\mbf{x};\mbf{f},\mbf{L})$ with \emph{explicit} main and error terms using analytic techniques in the spirit of Hal\'{a}sz' mean value theorem. In contrast, results in this direction have thus far been obtained by either using ergodic theoretic machinery, as in the works of Frantzikinakis and Host \cite{HF},\cite{HF2}, or, more recently, by using the nilpotent Hardy-Littlewood method of Green and Tao (see the recent paper of Matthiesen~\cite{Matth} for details). Neither of these papers give quantitative error terms.\\
For multiplicative functions $f,g : \mb{N} \ra \mb{U}$, we set
\begin{equation*}
\mb{D}(f,g;y,x) := \left(\sum_{y < p \leq x} \frac{1-\text{Re}(f(p)\bar{g(p)})}{p}\right)^{\frac{1}{2}}
\end{equation*}
for $1 \leq y \leq x$, as well as $\mb{D}(f,g;x) := \mb{D}(f,g;1,x)$. We then define $\mb{D}(f,g;\infty) := \lim_{x \ra \infty} \mb{D}(f,g;x)$. We also put
\begin{equation*}
\mb{D}^{\ast}(f,g;y,x) := \left(\sum_{y < p^k \leq x} \frac{1-\text{Re}(f(p^k)\bar{g(p^k)})}{p^k}\right)^{\frac{1}{2}}.
\end{equation*}
For $Q,X \geq 1$, we shall write
\begin{equation*}
\mc{D}(g;X,Q) := \inf_{|t| \leq X; q \leq Q, \chi (q)} \mb{D}(f,\chi n^{it};X)^2,
\end{equation*}
where the infimum in $q$ is over all Dirichlet characters $\chi$ modulo $q$, for all $q \leq Q$. \\
Recently, using their deep structural theorem for multiplicative functions (see Theorem 2.1 in \cite{HF2}), Frantzikinakis and Host proved that for a vector of 1-bounded \emph{multiplicative} functions $\mbf{f}$ and a system of integral, affine linear forms,
\begin{equation}  \label{HF}
M(x;\mbf{f},\mbf{L})=cx^{iT}e(\omega(x))+o_{x\to\infty}(1),
\end{equation}
where $\omega:\mathbb{R}\to\mathbb{R}$ is a slowly-oscillating function and $c=0$ unless all of the functions $f_j$ are \emph{pretentious} in the sense that for each $1 \leq j \leq k$ there is a primitive Dirichlet character $\chi_j$ with modulus $q_j$, and $t_j \in \mb{R}$ such that $\mb{D}(f_j,\chi_jn^{it_j};\infty) < \infty$. In the latter case, they show that the parameter $T$ in \eqref{HF} depends in some way on $t_1,\ldots,t_k$ (for instance, when the system is primitive they prove that $T = t_1 + \cdots + t_k$). However, they do not give an explicit expression for $c$.\\
Our first result is a quantitative version of \eqref{HF}, with explicit main and error terms, in the case that \emph{all} of the functions $f_j$ are pretentious in the above sense. To state it, we need to introduce some notation and conventions.\\
Given a vector $\mbf{x} \in (0,\infty)^l$ we write $$\ell(\mbf{x}) := \sum_{1 \leq j \leq l} |x_j|.$$  We will also write $x_-$ and $x_+$ to denote, respectively, the minimum and maximum components of $\mbf{x}$. Given $A \geq 1$ and $B > 0$, we will say that a vector $\mbf{x} \in (0,\infty)^l$ is \emph{$(A,B)$-appropriate} if $x_- \geq 3$ and
$$
x_- > l\log_2((l+1)Ax_+)^2 (\log x_+)^B.
$$
This condition ensures that $\mbf{x}$ is not too skew. \\
%
For a system of linear forms $\mbf{L}$, we write $\mbf{L}(\mbf{0})$ to be the vector with components $L_j(\mbf{0})$, for $1 \leq j \leq k$. We also say that the \emph{height} of the system $\mbf{L}$ of affine linear forms is the maximum of the coefficients of all linear forms in $\mbf{L}$.\\
For any multiplicative function $f:\mathbb{N}\to\mathbb{U}$ and any prime $p$, we define the multiplicative function $f_p$ by
\begin{align}\label{localfactor}
 f_{p}(q^{\nu}) := \begin{cases} f(p^{\nu}) & \mbox{if } q=p \\ 1 & \mbox{if } q\ne p,\end{cases}.
 \end{align}
We then define the $p$-adic local average of $\mbf{f}$ on $\mbf{L}$ by
\begin{align*}
M_p(\mbf{f},\mbf{L}) := \lim_{x \ra \infty} x^{-l} \sum_{\mbf{n} \in [x]^l} \prod_{1 \leq j \leq k} f_{j,p}(L_j(\mbf{n})).
\end{align*}
For an integral vector $\mbf{a}=(a_1,\dots,a_k)$ and primitive Dirichlet characters $\chi_1,\ldots,\chi_k$ to respective moduli $q_1,\ldots,q_k,$ we set
\begin{align*}
\mc{I}(\mbf{x},\mbf{L},\mbf{t}) &:= \int_{[0,1]^l} \prod_{1 \leq j \leq k} L_j((u_1x_1,\ldots,u_lx_l))^{it_j}d\mbf{u}; \\
\Xi_{\mbf{a}}(\mbf{\chi},\mbf{L}) &:= \mathop{\sum_{b_1 (q_1)} \cdots \sum_{b_k (q_j)}}_{\exists \mbf{n} : L_j(\mbf{n})/a_j \equiv b_j (q_j) \forall j} \prod_{1 \leq j \leq k} \chi_j(b_j); \\
R(m_1,\cdots,m_k) &:= \lim_{x \ra \infty} x^{-l}\sum_{\mbf{n} \in [x]^l \atop m_j|L_j(\mbf{n}) \forall j} 1,
\end{align*} 
and 
$$C_{\mbf{a}}(\mbf{x},\mbf{\chi},\mbf{t},\mbf{L}):= R(q_1a_1,\ldots, q_ka_k)\Xi_{\mbf{a}}(\mbf{\chi},\mbf{L})\mc{I}(\mbf{x},\mbf{L},\mbf{t}).$$
Finally, we recall that the \emph{radical} of a positive integer $n$ is $\rad(n) := \prod_{p|n} p.$ \\
We begin by stating one corollary of our main theorem.
\begin{cor} \label{CORPRET}
Let $A, q\geq 2$, $B > 0$, and let $\mbf{x} \in (0,\infty)^{l}$ be $(A,B)$-appropriate. Let $\mbf{f} = (f_1,\ldots,f_k)$ be a vector of 1-bounded multiplicative functions. Let $\mbf{L}$ be a primitive integral system of $k$ affine linear forms in $l$ variables with height at most $A$. \\
Suppose that  there are primitive Dirichlet characters $\chi_1,\ldots,\chi_k$ modulo $q$ and $t_1,\ldots,t_k\in \mb{R}$ such that $\mathbb{D}(f_j(n),\chi_jn^{it_j})<\infty$ for all $1\le j\le k.$ Let $F_j(n) := f_j(n)\bar{\chi}_j(n)n^{-it_j}$. Put $X := \ell(\mbf{x}) + 1.$ 
Then
\begin{align*}
&M(\mbf{x};\mbf{f},\mbf{L}) =\left(\sum_{\rad(a_j)|q \atop \forall 1 \leq j \leq k} \prod_{1 \leq j \leq k} \frac{f_j(a_j)}{a_j^{it_j}}C_{\mbf{a}}(\mbf{x},\mbf{\chi},\mbf{t},\mbf{L})\right)\prod_{p \leq AX \atop p\nmid q} M_p(\mbf{F},\mbf{L})+o(1).\end{align*}
\end{cor}
More generally we have the following fully explicit result.
\begin{thm}\label{MultAvg}
Let $A \geq 2$, $B > 0$, and let $\mbf{x} \in (0,\infty)^{l}$ be $(A,B)$-appropriate. Let $\mbf{f} = (f_1,\ldots,f_k)$ be a vector of 1-bounded multiplicative functions. Let $\mbf{L}$ be a primitive integral system of $k$ affine linear forms in $l$ variables with height at most $A$. \\
Fix a set of primitive Dirichlet characters $\chi_1,\ldots,\chi_k$ to respective moduli $q_1,\ldots,q_k$, and $t_1,\ldots,t_k\in \mb{R}.$ Let $F_j(n) := f_j(n)\bar{\chi}_j(n)n^{-it_j}$. Put $X := \ell(\mbf{x}) + 1$ and let $\max_{1 \leq j \leq k} q_j < y \leq X$. If $q_j = q$ for all $j$ then
\begin{align*}
&M(\mbf{x};\mbf{f},\mbf{L}) = \left(1+O_{k,l}\left(\frac{1}{\log y}\right)\right)\left(\sum_{\rad(a_j)|q \atop \forall 1 \leq j \leq k} \prod_{1 \leq j \leq k} \frac{f_j(a_j)}{a_j^{it_j}}C_{\mbf{a}}(\mbf{x},\mbf{\chi},\mbf{t},\mbf{L})\right)\prod_{p \leq AX \atop p\nmid q} M_p(\mbf{F},\mbf{L}) \\
&+ O_{k,l}\left(\prod_{p|q}\left(1-\frac{1}{\sqrt{p}}\right)^{-1} \left(\sum_{1 \leq j \leq k} \mb{D}^{\ast}(f_j,\chi_jn^{it_j}; y,AX)+ \frac{1}{(\log X)^{B'}} \right)\right) \\
&+ O_{k,l}\left(\frac{1}{x_-}\left(A+ q^ke^{\frac{3ky}{\log y}}\left(\sum_{\text{rad}(a_j)|q \atop \forall 1 \leq j \leq k} [a_1,\ldots,a_k]^{-1}\right)\prod_{1 \leq j \leq k} \max\{1,|t_j|\}\right)+ \frac{(\log y)^2}{\sqrt{y}}\right),
\end{align*}
where $B' := \min\{1,B/2\}.$ More generally, for any collection of moduli $q_j$,
\begin{align}
&M(\mbf{x};\mbf{f},\mbf{L}) = \left(1+O_{k,l}\left(\frac{1}{\log y}\right)\right) \left(\sum_{\rad(a_j)|q_j \atop \forall 1 \leq j \leq k} \prod_{1 \leq j \leq k} \frac{f_j(a_j)}{a_j^{it_j}} C_{\mbf{a}}(\mbf{x},\mbf{\chi},\mbf{t},\mbf{L})\right)\mc{S}_{\mbf{a}}(y;\mbf{f},\mbf{L}) \prod_{y < p \leq AX} M_p(\mbf{F},\mbf{L})\label{MAIN2}\\
&+ O_{k,l}\left(\sum_{1 \leq j \leq k} \prod_{p|q_j}\left(1-\frac{1}{\sqrt{p}}\right)^{-1}\left(\mb{D}^{\ast}(f_j,\chi_jn^{it_j}; y, AX)+ \frac{1}{(\log X)^{B'}}\right)\right) \nonumber\\
&+ O_{k,l}\left(\frac{1}{x_-}\left(A + e^{\frac{3ky}{\log y}}\left(\sum_{\text{rad}(a_j)|q_j \atop \forall 1 \leq j \leq k} [a_1,\ldots,a_k]^{-1}\right)\prod_{1 \leq j \leq k} q_j\max\{1,|t_j|\}\right) +\frac{(\log y)^2}{\sqrt{y}}\right), \label{ERROR2}
\end{align}
where, for $\mbf{a}, \mbf{d} \in \mb{N}^k$,
\begin{equation*}
R_{\mbf{a},\mbf{d}}(\mbf{L};\mbf{u},\mbf{v}) := \lim_{x \ra \infty} x^{-l} \sum_{\mbf{n} \in [x]^l \atop L_j(\mbf{n})/a_j \equiv u_j(q_j),L_j(\mbf{n}) \equiv v_j(a_jd_j) \forall j} 1
\end{equation*}
and 
\begin{equation*}
\mc{S}_{\mbf{a}}(y;\mbf{f},\mbf{L}) := R(q_1a_1,\ldots,a_kq_k)^{-1}\sum_{P^+(d_j) \leq y \atop (d_j,q_j) = 1 \forall j} R_{\mbf{a},\mbf{d}}(\mbf{L}-\mbf{L}(\mbf{0}),\mbf{0},\mbf{0})\prod_{1 \leq j \leq k} (\mu \ast F_j)(d_j).
\end{equation*}
\end{thm}
Theorem~\ref{MultAvg} shows that a \emph{local-to-global} phenomenon occurs for correlations of multiplicative functions, i.e., the global average correlation is the product of the local average correlations, determined by the functions $f_{j,p}$ and the characters $\chi_j$ and $n\mapsto n^{it_j}$. Our proof of Theorem~\ref{MultAvg} generalizes and extends the ideas from \cite{Klu}.
\begin{rem} \label{REMXS}
Note that when $x_j = x$ for all $j$, we have $\mc{I}(\mbf{x},\mbf{L},\mbf{t}) = x^{iT}\mc{I}(\mbf{L},\mbf{t})$, where $T := \sum_{1 \leq j \leq k} t_j$ and 
\begin{equation*}
\mc{I}(\mbf{L},\mbf{t}) := \int_{[0,1]^l} \prod_{1 \leq j \leq k} L_j(\mbf{u})^{it_j} d\mbf{u}.
\end{equation*}
This is consistent with the result in \cite{HF} mentioned above.
\end{rem}
\begin{rem}\label{MULTSEVVAR}
The distinction between the case in which the $q_j$ are all equal and the case in which they are not stems from the fact that the Chinese Remainder Theorem implies that $R(m_1,\ldots,m_k)$ is only multiplicative, and not \emph{firmly multiplicative} (see Section 3.2 of \cite{TOT}). That is, $R$ satisfies the identity $$R(m_1n_1,\ldots,m_kn_k) = R(m_1,\ldots,m_k)R(n_1,\ldots,n_k)$$ whenever $(m_1\cdots m_k,n_1\cdots n_k) = 1$, but in general it is \emph{not} sufficient that $(m_j,n_j) = 1$ for all $j$. This nuance concerning multiplicative functions in several variables (which is manifest in \eqref{RDENSMULT} below) prevents us from getting a conclusion that is uniform over all fixed moduli $q_j$.
\end{rem}
\begin{rem} \label{SQUAREFREE}
The error term in \ref{MultAvg} can be improved in a number of ways when the functions $f_j$ satisfy certain natural restrictions. For example, it follows from the proof of Theorem \ref{MultAvg} that the term $\frac{(\log y)^2}{\sqrt{y}}$ can be replaced by $y^{-1+o(1)}$ when each $f_j$ is supported on squarefree integers, and when the $f_j$ are all completely multiplicative we can replace $(\log X)^{-B'}$ by $(\log X)^{-B/2}$.
\end{rem}
When at least one of the functions $f_j$ is \emph{non-pretentious}, we are able to recover quantitative versions of the results from~\cite{HF} whenever $k \leq 3$ or the linear forms $L_j,$ $1\le j\le k$ are \emph{sufficiently} linearly independent. This independence is measured by \emph{Cauchy-Schwarz complexity} (see the end of Section 2 for a definition). 
\begin{prop}\label{MultAvgNonPret}
Let $A \geq 1$. Let $\mbf{f} = (f_1,\ldots,f_k)$ be a vector of multiplicative functions $f_j: \mb{N} \ra \mb{U}$. Let $\mbf{L}$ be a primitive integral system of affine linear forms in $l$ variables with height at most $A$ and Cauchy-Schwarz complexity at most 1.
Then there are absolute constants $c_1,c_2 > 0$ such that if, for some $1 \leq j_0 \leq k$, we have $\mc{D}_{j_0}(x) := \mc{D}(f_{j_0}; 10Ax, (\log x)^{1/125} ) \ra \infty$ as $x \ra \infty$, 
\begin{equation*}
M(x;\mbf{f},\mbf{L}) \ll_{k,l,A}e^{-c_1\mc{D}_{j_0}(x)} + (\log x)^{-c_2}.
\end{equation*}
\end{prop}
This result is a consequence of the recent work of Matomaki, Radziwi\l{}\l{} and Tao  on the averaged Elliott conjecture (see Theorem 1.6 of \cite{MRT}).
\subsection{Application: Gowers Norms of 1-Bounded Multiplicative Functions}
One motivation for investigations regarding affine linear averages of multiplicative functions comes from the study of Gowers norms. Let $(G,+)$ be a finite Abelian group, and let $f : G \ra \mb{C}$ be a map. Write
\begin{equation*}
\mb{E}_{x \in G}(f) := |G|^{-1}\sum_{x \in G} f(x),
\end{equation*}
and $\mb{E}_{x_1,\ldots,x_{k+1} \in G}(f) = \mb{E}_{x_{k+1} \in G} \mb{E}_{x_1,\ldots,x_k \in G} (f)$. For each $k \geq 1$ we define the $U^k(G)$-\emph{Gowers norm} of $f$ via
\begin{equation*}
\|f\|^{2^k}_{U^k(G)} := \mb{E}_{x,h_1,\ldots,h_k \in G} \prod_{\mbf{s} \in \{0,1\}^k} \mc{C}^{|\mbf{s}|} f(x+\mbf{s} \cdot \mbf{h}),
\end{equation*}
where, given a vector $\mbf{s} \in \{0,1\}^k$ we write $|\mbf{s}| = \sum_{1 \leq j \leq k} s_j$, and $\mc{C} : \mb{C}^{|G|} \ra \mb{C}^{|G|}$ is the conjugation operator $\mc{C}(g) = \bar{g}$. Gowers norms are fundamental in Additive Combinatorics as they provide a Fourier analytic framework for counting arithmetic progressions in groups. 
For background information regarding Gowers norms, see \cite{HOFA}. \\
We can extend Gowers norms to maps on intervals $[1,x] \subset \mb{N}$ as follows: let $N > x$ be a sufficiently large prime and let $G = \mb{Z}/N\mb{Z}$. Then the Gowers norm of a map $f: \mb{N} \ra \mb{C}$ on $[1,x]$ is given by
\begin{equation*}
\|f\|_{U^k(x)} := \|f1_{[1,x]}\|_{U^k(\mb{Z}/N\mb{Z})}/\|1_{[1,x]}\|_{U^k(\mb{Z}/N\mb{Z})},
\end{equation*}
where $1_{[1,x]}$ is the characteristic function of the interval $[1,x]$ as a subset of $\mb{Z}/N\mb{Z}$.
\begin{def1}
Let $k \geq 2$ and $K := 2^{k}$. The \emph{Gowers system (of order $k$)} is the system $\mbf{L}_k := \{L_j\}_{1 \leq j \leq K}$ of $k$-ary homogeneous linear forms such that if the binary expansion of $j$ is $\sum_{0 \leq l \leq k-1} \alpha_l2^l \leq K$ then 
\begin{equation*}
L_j(n_1,\ldots,n_{k+1}) = n_{k+1} + \sum_{0 \leq l\leq k-1} \alpha_jn_{j+1}.
\end{equation*}
\end{def1}
Note that for $f$ multiplicative, if $j = \sum_{0 \leq l \leq k-1} \alpha_l2^l$ and $d_j := \sum_{0 \leq l \leq k-1} \alpha_l$ then with $f_j := \mc{C}^{d_j} f$, we have $\|f\|_{U^k(x)}^{2^k} = M(x;\mbf{f},\mbf{L}_k)$. Theorem \ref{MultAvg} thus indeed furnishes estimates for Gowers norms of multiplicative functions. \\
Consider the $U^k(x)$ norm of a multiplicative function $f$ such that for some primitive character $\chi$ with conductor $q$ and a real number $t$ we have $\mb{D}(f,\chi n^{it}; \infty) < \infty$. With the notation above, our correlation has the form
\begin{equation*}
\|f\|_{U^k(x)}^{2^k} := x^{-(k+1)} \sum_{\mbf{n} \in [x]^{k+1}}\left( \prod_{1 \leq j \leq K \atop d_j \text{ even}} f(L_j(\mbf{n})) \right) \bar{\left(\prod_{1 \leq j \leq K \atop d_j \text{ odd}} f(L_j(\mbf{n}))\right)},
\end{equation*}
and Theorem \ref{MultAvg} applies. The Dirichlet character factor takes the form
\begin{equation*}
\Xi_{k,\mbf{a}}(\chi) := \Xi_{\mbf{a}}((\mc{C}^{d_1} \chi,\ldots,\mc{C}^{d_K}\chi),\mbf{L}_k)=\mathop{\sum_{b_1(q)} \cdots \sum_{b_K (q)}}_{\exists \mbf{n} : L_j(\mbf{n})/a_j \equiv b_j (q) \forall j} \chi\left(\prod_{1 \leq j \leq K \atop d_j \text{ even}} b_j\right)\bar{\chi\left(\prod_{1 \leq j \leq K \atop d_j \text{ odd}} b_j\right)},
\end{equation*}
while the Archimedean character factor is
\begin{equation*}
I_k(t) := I(\mbf{L}_k,((-1)^{d_1}t,\ldots,(-1)^{d_K}t)) = \int_{[0,1]^{k+1}} \left(\prod_{1 \leq j \leq K \atop d_j \text{ even}} L_j(\mbf{u})\right)^{it} \left(\prod_{1 \leq j \leq K \atop d_j \text{ odd}} L_j(\mbf{u})\right)^{-it}d\mbf{u},
\end{equation*}
for Archimedean characters. \\
The local-to-global principle for Gowers norms is thus as follows.
\begin{cor}\label{gowersassymp}
Let $f:\mathbb{N}\to\mathbb{C}$ be a 1-bounded multiplicative function. Let $k\ge 2$ and put $K := 2^k$.\\
i) If $\mathbb{D}(f,n^{it}\chi,\infty)=\infty$ for all Dirichlet characters $\chi$ and all $t\in\mathbb{R},$ then
\[\|f\|_{U^2(x)}=o_{x\to\infty}(1).\]\\
ii) If there exists a primitive Dirichlet character $\chi$ of conductor $q$ and $t\in\mathbb{R}$ such that  $\mathbb{D}(f,n^{it}\chi,\infty)<\infty,$ then
\begin{align*}
\|f\|_{U^k(x)}^{2^k}&=I_k(t)\left(\sum_{\rad(a_j)|q \atop \forall 1 \leq j \leq K} \prod_{1 \leq j \leq K} \mc{C}^{d_j}\left(\frac{f(a_j)}{a_j^{1+it}}\right) R(qa_1,\cdots,qa_K)\Xi_{k,\mbf{a}}(\chi)\right)\prod_{p\le x,\ p\nmid q}\|f_p(n)\bar{\chi(n)} n^{-it}\|_{U^k(x)}^{2^k}\\
&+O\left(\mathbb{D}(1,f(n)\bar{\chi(n)}n^{-it};\log x;x)\right),
\end{align*}
where $d_j$ is the sum of the binary digits of $j$.
\end{cor}
\subsection{Application: Sign Changes of Multiplicative Functions in 3- and 4-term Arithmetic Progressions}
Let $\lambda$ denote the Liouville function $\lambda(n) := (-1)^{\Omega(n)}$, where $\Omega(n)$ is the number of prime factors of $n$, counted with multiplicity. Chowla \cite{CHO} conjectured the following regarding sign patterns of $\lambda$.
\begin{conj}[Chowla for Sign Patterns] \label{CHOWSGN}
Let $k \geq 2$, let $\{h_1,\ldots,h_k\}$ be a sequence of distinct non-negative integers, and let $\mbf{\e} \in \{-1,1\}^k$ be a vector of signs. Then
\begin{equation*}
|\{n \leq x: \lambda(n+h_j) = \e_j \text{ for all } 1 \leq j \leq k\}| = \left(\frac{1}{2^k}+o(1)\right)x.
\end{equation*}
\end{conj}
In other words, it is expected that the vectors $(\lambda(n+h_1),\ldots,\lambda(n+h_k))$ are uniformly distributed among the $2^k$ possible patterns of $+$ and $-$ signs. Of particular interest is the case in which the forms $n\mapsto n+h_j$ constitute an arithmetic progression. This case requires that one understands the behaviour of a function sensitive to multiplicative structure on sets with additive structure. \\
Recently, lower density estimates for sign patterns of $\lambda$ of length 3 were given by Matom\"{a}ki, Radziwi\l \l \ and Tao \cite{MRTSigns}, and the exact logarithmic density of the set of $n$ yielding any fixed sign pattern of length 2 for $\lambda$ was obtained by Tao \cite{Tao}.\\
One may ask about the frequency of sign patterns for arbitrary multiplicative functions $f: \mb{N} \ra \{-1,1\}$ in place of $\lambda$, and on arithmetic progressions of length $3$ or more. Such questions have interest, for instance, because of their relationship with the distribution of quadratic non-residues modulo primes. \\
As an example of such investigations, Buttkewitz and Elsholtz \cite{ElB} recently classified those multiplicative sign functions that only yield a fixed length four sign pattern finitely often on certain 4-term APs.\\
We shall study several questions in this direction. We first consider fixed arithmetic progressions, giving explicit lower bounds for the upper logarithmic density of sign patterns of length 3 and 4 when $f$ is non-pretentious. In particular, we show the following. 
\begin{prop} \label{LOGDENSSPEC}
Let $d \geq 1$. Let $f: \mb{N} \ra \{-1,1\}$ be a non-pretentious multiplicative function. \\
i) Given $\mbf{\e} \in \{-1,1\}^3$, the upper density of the set of integers $n$ such that $$(f(n),f(n+d),f(n+2d)) = \mbf{\e}$$ is at least $\frac{1}{28}$. \\
ii) Given $\mbf{\e} \in \{-1,1\}^4$, the upper density of the set of integers $n$ such that $$(f(n),f(n+d),f(n+2d),f(n+3d)) = \mbf{\e} \text{ or } -\mbf{\e}$$ is at least $\frac{1}{28}$.
\end{prop}
We remark that the bound in~ \cite{MRTSigns} on the \emph{lower} density of length 3 sign patterns of $\lambda$ was {\it inexplicit}, due to the use of nonstandard analysis there. \\
Suppose $f$ is, in addition, completely multiplicative. As a consequence of Proposition \ref{LOGDENSSPEC}, we can determine an upper bound, uniform over all sign patterns, for the least $d$ required to find infinitely many 4-term AP's $(n,n+d,n+2d,n+3d)$ such that $(f(n),f(n+d),f(n+2d),f(n+3d))$. In the case of non-pretentious completely multiplicative functions, this improves on the work of Buttkewitz and Elsholtz (see Corollary 2.4 in \cite{ElB}).
\begin{cor}\label{GAPS4AP}
Let $f$ be a non-pretentious, completely multiplicative function, and let $p_0$ be the least prime for which $f(p_0) = -1$. Let $\mbf{\e} \in \{-1,1\}^4$ be a length 4 sign pattern, and let $d(\mbf{\e})$ denote the least $d$ such that for infinitely many $n$ we have $$(f(n),f(n+d),f(n+2d),f(n+3d))= \mbf{\e}.$$ Then $d(\mbf{\e}) \leq p_0$. 
\end{cor}
We also consider corresponding questions about the natural density of sign patterns in \emph{almost all} progressions. We will establish the following equidistribution-type results for sign patterns of non-pretentious functions on \emph{almost} all 3-term APs in a suitable sense. In particular, we have an averaged analogue of Conjecture \ref{CHOWSGN} in this context. \\
For $c_1,c_2 > 0$ the constants in Proposition \ref{MultAvgNonPret}, define
\begin{equation} \label{RFX}
\mc{R}_f(x) := e^{-c_1\mc{D}(f;x,(\log x)^{\frac{1}{125}} )} + (\log x)^{-c_2}.
\end{equation}
\begin{thm} \label{EQUIDIST}
Let $f: \mb{N} \ra \{-1,1\}$ be multiplicative. Let $\mbf{\e} \in \{-1,1\}^3$. Except for $O\left(x\mc{R}_f(x)^{\frac{1}{3}}\right)$ choices of $d \leq x$, we have
\begin{equation*}
|\{n \leq x: f(n+jd) = \e_j \text{ for all } 0 \leq j \leq 2\}| = \left(\frac{1}{8} + O\left(\mc{R}_f(x)^{\frac{1}{3}}\right)\right)x.
\end{equation*}
\end{thm}
Finally, we establish an analogue of Theorem \ref{EQUIDIST} when $f$ is pretentious and investigate to what extent this average density can be biased away from the density predicted by equidistribution. In so doing, we establish a quantitative refinement of the results of Buttkewitz and Elsholtz \cite{ElB}. See Remark \ref{BUTELS} for a discussion of the connection between our results and those of \cite{ElB}.\\
It turns out that when $f$ is pretentious to a real primitive character $\chi$ with conductor $q$, $f$ behaves well on arithmetic progressions with difference $d$ not divisible by $q$. 
\begin{thm} \label{AE}
Let $\delta > 0$ and let $2 \leq (\log x)^{\delta} \leq z \leq x$, with $z = o(x)$. Let $\chi$ be a real primitive character with conductor $q$, where $q$ is coprime to 6. Let $f: \mb{N} \ra \{-1,1\}$ be a multiplicative function with $\mb{D}(f,\chi; \infty) < \infty$, and $\mbf{\e} \in \{-1,1\}^4$. Then for all but $o(z)$ integers $d \leq z$ not divisible by $q$, we have
\begin{equation}
|\{n \leq x : f(n+jd) = \e_j \text{ for each } 0 \leq j \leq 3\}|  = \left(\frac{1}{16} + o(1)\right)x.
\end{equation}
In particular, if $q \geq 5$ and coprime to 6 then a positive proportion of the length 4 arithmetic progressions in $[1,z]\times [1,x]$ exhibit the sign pattern $\mbf{\e}$.
\end{thm}
We note that the restriction $(q,6) = 1$ is merely technical, and could be removed with more effort. With additional effort we could also quantify the size of the exceptional set in Theorem \ref{AE}; we have chosen not to do this in order to avoid making this paper even longer.\\
On the other hand, when the shifts $d$ are divisible by $q$, the behaviour is much more erratic. In fact, for such arithmetic progressions there can be a bias, as is evident from the following theorems. To state them, we require additional notation. \\
Given $r \in \mb{N}$, set $[r] := \{0,\ldots,r\}.$ For $S\subseteq [r],$ we write $$\mbf{L}_{S} := \{(n,d) \mapsto n+jd : j \in S\},$$ and for each pair of sets $S,T\subseteq [r]$ we associate the system of forms $$\mbf{L}_{S,T} := \{(n,n',d) \mapsto n+jd : j \in S\} \cup \{(n,n',d) \mapsto n'+j'd : j' \in T\}.$$
For each $\lambda \in \mb{Z}$ and $p|q$, define $\mb{E}_{\lambda}/\mb{F}_p$ to be the elliptic curve over $\mb{F}_p$ with Legendre model $$E_{\lambda} : y^2 \equiv x(x-1)(x-\lambda) \ (p).$$ Let $b$ denote a reduced element of the residue class inverse to 2 modulo $q$, and set $\Delta_p := p+1-\#E_{3b^2}(\mb{F}_p)$. Finally, put
\begin{equation} \label{BIASMEAN}
A_{\mbf{\e}}(f;q) := \e_0\e_1\e_2\e_3 \prod_{p|q}\frac{\mu(p)\Delta_p}{p+1} \prod_{p\nmid q} M_p(f\chi\mbf{1}_4,\mbf{L}_{[3]}).
\end{equation}
\begin{thm} \label{PRETPATS}
Let $\delta > 0$ and let $2 \leq (\log x)^{\delta} \leq z \leq x$, and $z = o(x)$. Let $\chi$ be a real primitive character with modulus $q$, with $q$ coprime to 6. Let $f: \mb{N} \ra \{-1,1\}$ be a multiplicative function with $\mb{D}(f,\chi; \infty) < \infty$.  For any $\mbf{\e} \in \{-1,1\}^4$,
\begin{align}
\frac{1}{xz}\sum_{d \leq z} |\{n \leq x : f(n+jd) = \e_j \forall j\}| = \frac{1}{16}\left(1+A_{\mbf{\e}}(f;q)\right) + o(1). \label{NEAT1}
\end{align}
\end{thm}
\begin{rem} \label{NONVAN}
Put $\lambda = 3b^2$, where $b$ is as above. The role that the elliptic curve $E_{\lambda}$ plays in this problem stems from the complete character sum yielded by the character local factor in Corollary \ref{CORPRET}, taking account of the compatibility conditions imposed on its summands.  Note that $q$ is necessarily squarefree, being the odd conductor of a real character. By the Chinese Remainder Theorem, the complete sum over $\mb{Z}/q\mb{Z}$ splits as a product of complete sums of Legendre symbols of cubic polynomials over $\mb{Z}/p\mb{Z}$, which is then easily related to point counts for elliptic curves over $\mb{F}_p$.\\
 The quantity $\Delta_p$, which is the trace of the Frobenius element of $E_{\lambda}$ over $\mb{F}_p$, is non-zero if, and only if, the curve $E_{\lambda}$ is not supersingular over $\mb{F}_p$ (see Exercise V.5.10 of \cite{Sil}). Since the set of primes at which an elliptic curve is supersingular is typically small (e.g., for non-CM elliptic curves, see Theorem V.4.7) we expect that if $E_{\lambda}$ is generic with respect to each of the primes dividing $q$ then the product in \eqref{NEAT1} is non-vanishing, and a bias exists according to the sign of $\e_0\e_1\e_2\e_3$. \\
For concreteness, we may note that if $q$ is composed solely of primes $p \equiv 1 (4)$ (i.e., $q$ is an odd sum of two squares) then the product is non-zero. Indeed, note that the points $(1,0)$ and $(\lambda,0)$ are both trivially 2-torsion on $E_{\lambda}(\mb{F}_p)$. As such, $E_{\lambda}(\mb{F}_p)$ contains a subgroup of order 4. Hence, $$\Delta_p \equiv p+1 \ (4) \equiv 2 \ (4),$$ 
so that $\Delta_p \geq 2$ for all $p|q$.
\end{rem}
We also compute the mean-squared deviation. For a discussion regarding the size of the deviation in \eqref{MSQ}, including an heuristic for why it should generally be $\Omega(1)$, see Remark \ref{GENERIC} below.  
\begin{thm}\label{WITHQTHM}
With the hypotheses in Theorem \ref{PRETPATS},
\begin{align}
&\frac{1}{z}\sum_{d \leq z} \left(x^{-1}|\{n \leq x : f(n+jd) = \e_j \forall \ 0 \leq j \leq 3\}| - \frac{1}{16}\left(1+A_{\mbf{\e}}(f;q)\right)\right)^2 \nonumber\\
&= \frac{1}{256}\left((T_{4,4} - A_{\mbf{\e}}(f;q)^2) + 2\e_0\e_1\e_2\e_3\left(\sum_{0 \leq i < j \leq 3} \e_i\e_j\right)T_{4,2} + \left(\sum_{0 \leq i < j \leq 3} \e_i\e_j\right)^2T_{2,2}\right) + o(1), \label{MSQ}
\end{align}
where we have set
\begin{align*}
T_{2,2} &:= \prod_{p\nmid q} M_p(f\chi \mbf{1}_4,\mbf{L}_{[1],[1]})\prod_{p|q} \frac{p}{p^2+p+1} \\
T_{4,2} &:= \prod_{p\nmid q} M_p(f\chi \mbf{1}_6,\mbf{L}_{[3],[1]})\prod_{p|q}\frac{(p-\Delta_p)(p+1)-\Delta_p}{p^2(p+1)} \\
T_{4,4} &:= A_{\mbf{\e}}(f;q)^2\prod_{p|q}\frac{(1+1/\Delta_p)^2 + 1/p + p/\Delta_p^2}{1+1/p(p+1)}.
\end{align*}
\end{thm}
\begin{rem}
When $d$ is a multiple of $q$, the contribution to the sign given by $\chi$ on $(n,n+d,n+2d,n+3d)$ is completely determined by $n$, and since $f\chi$ is 1-pretentious this means that $f\chi$ \emph{should} only change sign infrequently on 4-term arithmetic progressions with difference $d$. As such, we heuristically expect that certain sign patterns (depending on $\chi$) occur more often than others among the vectors $(f(n),f(n+d),f(n+2d),f(n+3d))$, an intuition that is confirmed by Theorems \ref{PRETPATS} and \ref{WITHQTHM}.
\end{rem}
\begin{rem}\label{BUTELS}
It is worthwhile mentioning how the results of this paper relate to the results in \cite{ElB}. In the latter paper, it is shown that, except for two explicit collections of multiplicative functions $f: \mb{N} \ra \{-1,1\}$, any $f$ takes on each length 4 sign pattern on infinitely many 4-term arithmetic progressions. The counterexamples are of one of the following two types: 
\begin{enumerate}
\item there is a prime $p$ such  for all $\nu \geq 1$, $f(q^{\nu}) = 1$ if $q \neq p$, while $f(p^{\nu}) = (-1)^{\nu}$;
\item $f(n) = \chi_3(n)$ for all $(n,3) = 1$, where $\chi_3$ is the primitive real character modulo $3$.
\end{enumerate}
In each of these two cases, certain sign patterns can never be exhibited on length 4 arithmetic progressions. For functions of the first type, for example, the sign patterns $(1,1,1,-1)$ and $(1,-1,-1,-1)$ only occur on finitely many length 4 arithmetic progressions. \\
Note that the functions of both of these types are necessarily pretentious. In the first case, they are pretentious to the trivial character, with conductor $q = 1$, while in the second they are pretentious to $\chi_3$. This latter example is excluded from the above analysis, so consider instead the $1$-pretentious examples. \\
In this case, all common differences $d$ are divisible by the conductor. Hence, Theorem \ref{AE} does not apply to any $d$, and the irregularity of distribution in Theorem \ref{WITHQTHM} is necessary (notice that the examples of sign patterns given above are both such that $\e_0\e_1\e_2\e_3 = -1$, as we should expect from Theorem \ref{PRETPATS}). Conversely, if a function is pretentious to a real character with conductor strictly greater than 1 then it follows from Theorem \ref{AE} that for any given sign pattern we can find infinitely many arithmetic progressions giving an instance of this sign pattern. Thus, the theorems of this section are consistent with, and quantitatively refine, the results of Buttkewitz and Elsholtz.
\end{rem}
\section*{Acknowledgments}
We thank John Friedlander and Andrew Granville for their encouragement and advice in improving the exposition of the paper. We would also like to thank Christian Elsholtz for suggesting the problem of sign patterns to us.
\section{Auxiliary Results}
In this section we collect some lemmata to be used in the remainder of the paper. 
\subsection{Technical Results for Theorem~\ref{MultAvg}.}
For the proof of Theorem \ref{MultAvg}, we will need several technical results.  The first is a version of the Tur\'{a}n-Kubilius inequality applicable to additive functions whose arguments are integral affine linear forms in several variables. While the proof of this result is fairly routine, the authors could not find it in the literature. We therefore give a full proof here for completeness. We first require some definitions. \\
For a primitive, integral affine linear form $L : \mb{R}^l \ra \mb{R}$, let
\begin{equation*}
\omega_L(p^k) := |\{\mbf{b} \in (\mb{Z}/p^k \mb{Z})^l : p^k||L(\mbf{b})\}|.
\end{equation*}
Furthermore, given an additive function $h : \mb{N} \ra \mb{C}$, put
\begin{align*}
\mu_{h,L}(x) &:= \sum_{p^k \leq x} h(p^k) \left(\frac{\omega_L(p^k)}{p^{kl}} - \frac{\omega_L(p^{k+1})}{p^{(k+1)l}}\right); \\
\sigma_{h,L}(x)^2 &:= \sum_{p^k \leq x} |h(p^k)|^2 \left(\frac{\omega_L(p^k)}{p^{kl}} - \frac{\omega_L(p^{k+1})}{p^{(k+1)l}}\right).
\end{align*}
\begin{rem} \label{SIMP}
Note that we can lift a solution to the congruence $L(\mbf{b}) \equiv 0 (p^k)$ to precisely $\omega_{L-L(\mbf{0})}(p)$ distinct solutions mod $p^{k+1}$ via $b_j' := r_jp^k + b_j$ whenever the vector $\mbf{r}$ satisfies $L(\mbf{r}) -L(\mbf{0}) \equiv 0 (p)$. Moreover, since $L$ is primitive there is some index $1 \leq j_0 \leq k$ such that the coefficient $\alpha_{j_0}$ satisfies $(\alpha_{j_0},p) = 1$. Thus, given any choice of $r_j$ for $j \neq j_0$, there is a unique $r_{j_0}$ mod $p$ such that the congruence $L(\mbf{r}) -L(\mbf{0}) \equiv 0 (p)$ is satisfied. Hence, $\omega_{L-L(\mbf{0})}(p) = p^{l-1}$, and by induction, we have $\omega_L(p^{\mu}) = \omega_{L-L(\mbf{0})}(p)^{\mu} = p^{\mu(l-1)}$. Thus, we can rewrite $\mu_{h,L}$ and $\sigma_{h,L}^2$ as
\begin{align}
\mu_h(x) &= \mu_{h,L}(x) = \sum_{p^k \leq x} \frac{h(p^k)}{p^k}\left(1-\frac{1}{p}\right); \label{ALTMU}\\
\sigma_h(x)^2 &= \sigma_{h,L}(x)^2 = \sum_{p^k \leq x} \frac{|h(p^k)|^2}{p^k} \left(1-\frac{1}{p}\right). \label{ALTSIG}
\end{align}
\end{rem}
Write $X := \ell(\mbf{x}) + 1$, for $\mbf{x} \in (0,\infty)^l$.
\begin{lem} \label{erdoskacpol}
Let $A \geq 1$ and $\mbf{x} \in [1,\infty)^l$. Let $h: \mb{N} \ra \mb{C}$ be an additive function satisfying $|h(p^k)| \ll 1$ uniformly on prime powers $p^k$, and suppose that $L$ is a primitive integral affine linear form in $l$ variables with height at most $A$. Then
\begin{equation} \label{TKLin}
\langle \mbf{x}\rangle^{-1} \sum_{\mbf{n} \in \mc{B}(\mbf{x})} |h(L(\mbf{n}))-\mu_{h}(AX)|^2 \ll_l \sigma_h(AX)^2+\frac{|\mu_h(AX)|}{x_-}.
\end{equation}
Thus, if $f$ is a 1-bounded multiplicative function and $h$ is the additive function defined by $h(p^k) = f(p^k) -1$, and $\mbf{x}$ is $(A,B)$-appropriate then
\begin{equation} \label{MULTCASE}
\langle \mbf{x}\rangle^{-1} \sum_{\mbf{n} \in \mc{B}(\mbf{x})} |h(L(\mbf{n}))- \mu_h(AX)|^2 \ll_l \mb{D}^{\ast}(1,f;AX)^2 + \frac{1}{(\log X)^B}.
\end{equation}
\end{lem}
\begin{proof}
Observe first that
\begin{align*}
\sum_{\mbf{n} \in \mc{B}(\mbf{x})} h(L(\mbf{n})) &= \sum_{p^k \leq AX} h(p^k) \sum_{\mbf{n} \in \mc{B}(x) \atop p^k || L(\mbf{n})} 1 \\
&= \sum_{p^k \leq AX} h(p^k) \left(\sum_{\mbf{b} \in (\mb{Z}/p^k \mb{Z})^l \atop L(\mbf{b}) \equiv 0 (p^k)} \sum_{\mbf{n} \in \mc{B}(x) \atop n_j \equiv b_j (p^k) \forall j} 1 - \sum_{\mbf{b} \in (\mb{Z}/p^k \mb{Z})^l \atop L(\mbf{b}) \equiv 0 (p^{k+1})} \sum_{\mbf{n} \in \mc{B}(x) \atop n_j \equiv b_j (p^{k+1}) \forall j} 1\right) \\
&= \langle \mbf{x}\rangle\left(1+O\left(x_-^{-1}\right)\right)\sum_{p^k \leq AX}h(p^k)\left(\frac{\omega_L(p^k)}{p^{kl}} - \frac{\omega_L(p^{k+1})}{p^{(k+1)l}}\right) \\
&= \langle \mbf{x}\rangle (1+O(x_-^{-1}))\mu_{h,L}(AX).
\end{align*}
Expanding the square in \eqref{TKLin}, we thus get
\begin{align}
&\langle \mbf{x}\rangle^{-1} \sum_{\mbf{n} \in \mc{B}(\mbf{x})} |h(L(\mbf{n})) - \mu_{h,L}(AX)|^2 \nonumber \\
&= \langle \mbf{x}\rangle^{-1} \sum_{\mbf{n} \in \mc{B}(\mbf{x})} |h(L(\mbf{n}))|^2 - 2\text{Re}\left(\bar{\mu_{h,L}(AX)} \langle \mbf{x}\rangle^{-1}\sum_{\mbf{n} \in \mc{B}(\mbf{x})} h(L(\mbf{n}))\right) + |\mu_{h,L}(AX)|^2 \nonumber\\
&= \langle \mbf{x}\rangle^{-1} \sum_{\mbf{n} \in \mc{B}(\mbf{x})} |h(L(\mbf{n}))|^2 - |\mu_{h,L}(AX)|^2 + O(|\mu_{h,L}(AX)|^2x_-^{-1}). \label{OPEN}
\end{align}
The first term in \eqref{OPEN} can be rewritten as
\begin{align*}
&\langle \mbf{x} \rangle^{-1} \sum_{\mbf{n} \in \mc{B}(\mbf{x})} |h(L(\mbf{n}))|^2 = \langle \mbf{x} \rangle^{-1} \sum_{\mbf{n} \in \mc{B}(\mbf{x})} \sum_{p^{\mu},q^{\nu} || L(\mbf{n})} h(p^{\mu})\bar{h(p^{\nu})} \\
&= \langle \mbf{x} \rangle^{-1} \sum_{\mbf{n} \in \mc{B}(\mbf{x})} \sum_{p^{\mu} || L(\mbf{n})} |h(p^{\mu})|^2 + \langle \mbf{x} \rangle^{-1} \sum_{\mbf{n} \in \mc{B}(\mbf{x})}  \sum_{p^{\mu},q^{\nu} || L(\mbf{n}) \atop p \neq q} h(p^{\mu}) \bar{h(q^{\nu})} \\
&= \left(1+O(x_-^{-1})\right)\sum_{p^{\mu} \leq AX} |h(p^{\mu})|^2 \left(\frac{\omega_L(p^{\mu})}{p^{\mu l}} - \frac{\omega_L(p^{\mu+1})}{p^{(\mu+1)l}}\right) + \langle \mbf{x}\rangle^{-1}\sum_{p^{\mu},q^{\nu} \leq AX \atop p\neq q} h(p^{\mu})\bar{h(q^{\nu})} \sum_{\mbf{n} \in \mc{B}(\mbf{x}) \atop p^{\mu},q^{\nu} || L(\mbf{n})} 1 \\
&= \left(1+O(x_-^{-1})\right) \left(\sigma_{h,L}^2(AX) + \sum_{p^{\mu},q^{\nu} \leq AX \atop p \neq q} h(p^{\mu})\bar{h(q^{\nu})} \left(\frac{\omega_{h,L}(p^{\mu})}{p^{\mu l}}-\frac{\omega_{h,L}(p^{\mu+1})}{p^{(\mu+1) l}}\right)\left(\frac{\omega_{h,L}(q^{\nu})}{q^{\nu l}}-\frac{\omega_{h,L}(q^{\nu+1})}{q^{(\nu+1) l}}\right)\right).
\end{align*}
The second term in \eqref{OPEN} can be expressed as
\begin{align*}
|\mu_{h,L}(AX)|^2 &= \sum_{p^{\mu},q^{\nu} \leq AX} h(p^{\mu})\bar{h(q^{\nu})} \left(\frac{\omega_{L}(p^{\mu})}{p^{\mu l}}-\frac{\omega_{L}(p^{\mu+1})}{p^{(\mu+1) l}}\right)\left(\frac{\omega_{L}(q^{\nu})}{q^{\nu l}}-\frac{\omega_{L}(q^{\nu+1})}{q^{(\nu+1) l}}\right) \\
&= \left(\sum_{p^{\mu},q^{\nu} \leq AX \atop p \neq q} + \sum_{p^{\mu}, q^{\nu} \leq AX \atop p = q} \right)h(p^{\mu})\bar{h(q^{\nu})} \left(\frac{\omega_{L}(p^{\mu})}{p^{\mu l}}-\frac{\omega_{L}(p^{\mu+1})}{p^{(\mu+1) l}}\right)\left(\frac{\omega_{h,L}(q^{\nu})}{q^{\nu l}}-\frac{\omega_{L}(q^{\nu+1})}{q^{(\nu+1) l}}\right).
\end{align*}
Subtracting these two expressions gives
\begin{align*}
&\langle \mbf{x} \rangle^{-1} \sum_{\mbf{n} \in \mc{B}(\mbf{x})} |h(L(\mbf{n}))|^2 - |\mu_{h,L}(AX)|^2 \\
&\ll \sigma_{h,L}(AX)^2 + |\mu_{h,L}(AX)|^2x-^{-1} \\
&+ \left|\sum_{p^{\mu} \leq p^{\nu} \leq AX}h(p^{\mu})\bar{h(p^{\nu})} \left(\frac{\omega_{h,L}(p^{\mu})}{p^{\mu l}}-\frac{\omega_{h,L}(p^{\mu+1})}{p^{(\mu+1) l}}\right)\left(\frac{\omega_{h,L}(p^{\nu})}{p^{\nu l}}-\frac{\omega_{h,L}(p^{\nu+1})}{p^{(\nu+1) l}}\right)\right|.
\end{align*}
Hence, by Cauchy-Schwarz and \eqref{ALTMU},
\begin{align*}
\langle \mbf{x} \rangle^{-1} \sum_{\mbf{n} \in \mc{B}(\mbf{x})} |h(L(\mbf{n}))-\mu_{h,L}(AX)|^2 
&\ll \sigma_{h,L}(AX)^2 + |\mu_{h,L}(AX)|^2x_-^{-1}.
\end{align*}
This prove \eqref{TKLin}. For \eqref{MULTCASE}, note that by \eqref{ALTSIG}
\begin{equation*}
\sigma_{h,L}(AX)^2 \ll \sum_{p^k \leq AX} \frac{|f(p^k)-1|^2}{p^k} \ll \sum_{p^k \leq AX} \frac{1-\text{Re}(f(p^k))}{p^k} = \mb{D}^{\ast}(1,f;AX)^2,
\end{equation*}
and that \eqref{ALTMU} together with the $(A,B)$-appropriateness condition imply that
$$|\mu_{h,L}(AX)|^2x_-^{-1} \ll \log_2((l+1)Ax_+)^2x_-^{-1} \leq \frac{1}{(\log x_+)^B} \ll_l \frac{1}{(\log(lx_+))^B} \leq \frac{1}{(\log X)^B}.$$
\end{proof}
For $1 \leq y \leq x$ and each $1 \leq j \leq k$, define
\begin{equation*}
\mathfrak{P}(f_j;y,x) := \prod_{y < p \leq x} \sum_{\nu \geq 0} f_j(p^{\nu})\left(\frac{\omega_{L_j}(p^{\nu})}{p^{\nu l}}-\frac{\omega_{L_j}(p^{\nu+1})}{p^{(\nu+1)l}}\right) = \prod_{y < p \leq x} \left(1-\frac{1}{p}\right)\left(1+\sum_{k \geq 1} \frac{f_j(p^k)}{p^k}\right),
\end{equation*}
and write $\mathfrak{P}(f_j;x) := \mathfrak{P}(f_j;1,x)$. The first representation for $\mathfrak{P}(f_j;y,x)$ will be useful later; the second one follows from Remark \ref{SIMP}. \\
The following lemma allows us to conveniently decompose multilinear averages of products of arithmetic functions with good error, provided that one of the sequences is multiplicative and 1-pretentious. 
\begin{lem}\label{key} 
Let $A\geq 2$, $q \geq 1$ and let $\mbf{x}$ be $(A,B)$-appropriate. Let $g : \mb{N}^l \ra \mb{U}$ be any sequence and let $f: \mb{N} \ra \mb{U}$ be a multiplicative function such that $f(n) = 1$ whenever $(n,q) > 1$. Also, let $L : \mb{R}^l \ra \mb{R}$ be a primitive, integral form with height at most $A$. Then
\[ \sum_{ \mbf{n}\in \mc{B}(\mbf{x}) \atop q|L(\mbf{n})}f(L(\mbf{n}))g(\mbf{n})=\mathfrak{P}(f;AX)\left(\sum_{\mbf{n}\in \mc{B}(\mbf{x}) \atop q|L(\mbf{n})}g(\mbf{n})\right)+O\left(\frac{\langle \mbf{x} \rangle}{\sqrt{q}}\left(\mathbb{D}^{\ast}(1,f;AX)+\frac{1}{(\log X)^{B'}}\right)\right),\]
where $B' := \min\{1,B/2\}$. 
\end{lem}
\begin{proof}
Since, for all $|z_j|,|w_j|\le 1$, $1 \leq j \leq n$, 
\begin{align*}
\left|\prod_{1\le j\le n}z_j-\prod_{1\le j\le n}w_j\right| &= \left|\prod_{1 \leq j \leq n-1} z_j(z_n-w_n) + w_n \left(\prod_{1 \leq j \leq n-1} z_j - \prod_{1 \leq j \leq n-1} w_j\right)\right| \\
&\leq |z_n - w_n| + \left|\prod_{1 \leq j \leq n-1} z_j - \prod_{1 \leq j \leq n-1} w_j\right|, 
\end{align*}
it follows by induction that
\begin{equation}\label{iterate1}
\left|\prod_{1\le j\le n}z_j-\prod_{1\le j\le n}w_j\right| \leq \sum_{1\le j\le n}|z_j-w_j|.
\end{equation}
Note that $e^{z-1}=z+O(|z-1|^2)$ for $|z|\le 1$. Therefore,
\begin{align*}
f(L(\mbf{n})) &= \prod_{p^k\vert| L(\bold{n})}f(p^k) = \prod_{p^{k}\vert| L(\bold{n})}e^{f(p^k)-1} + O\left(\left|\prod_{p^k\vert| L(\bold{n})}f(p^k)-\prod_{p^k\vert| L(\bold{n})}\left(f(p^k)+O(|f(p^k)-1|^2)\right)\right|\right)\\
&=\exp\left(\sum_{p^k\vert| L(\bold{n})}(f(p^k)-1)\right)+O\left(\sum_{p^k\vert| L(\bold{n})}|f(p^k)-1|^2\right).
\end{align*}
Define $h:\mb{N} \ra \mb{C}$ to be the additive function satisfying $h(p^k)=f(p^k)-1$ for each prime $p$ and $k \geq 1$. Note that $h(p^k) = 1$ whenever $p|q$. Hence,
\begin{align*}
\sum_{\bold{n}\in \mc{B}(\mbf{x})}f(L(\bold{n}))g(\bold{n})1_{q|L(\mbf{n})}-\sum_{\bold{n}\in \mc{B}(\mbf{x})}&g(\bold{n})e^{h(L(\bold{n}))}1_{q|L(\mbf{n})}\ll \sum_{\bold{n}\in \mc{B}(\mbf{x})}1_{q|L(\mbf{n})}\sum_{\substack{p^k\vert| L(\bold{n}),\\ p^k\leq AX}}|h(p^k)|^2&\\&\ll \langle \mbf{x}\rangle \sum_{p^k\leq AX \atop p|q}\frac{|f(p^k)-1|^2}{[q,p^k]}\ll \frac{1}{q}\langle \mbf{x}\rangle \mathbb{D}^{\ast}(f,1;AX)^2.
\end{align*}
Since $|e^a-e^b|\ll|a-b|$ for $\operatorname{Re}{(a)},\operatorname{Re}{(b)}\le 0,$ Cauchy-Schwarz together with Lemma \ref{erdoskacpol} imply 
\begin{align*}
&\sum_{\bold{n}\in \mc{B}(\mbf{x})}g(\bold{n})e^{h(L(\bold{n}))}1_{q|L(\mbf{n})}-e^{\mu_{h}(AX)}\sum_{\bold{n}\in \mc{B}(\mbf{x})}g(\bold{n})1_{q|L(\mbf{n})}\\
&\ll \sum_{\bold{n}\in \mc{B}(\mbf{x})}|e^{h(L(\bold{n}))}-e^{\mu_{h}(AX)}|1_{q|L(\mbf{n})}\ll \sum_{\bold{n}\in \mc{B}(\mbf{x})}|h(L(\bold{n}))-\mu_{h}(X)|1_{q|L(\mbf{n})}\\&\le \left(\frac{\langle \mbf{x}\rangle}{q}\sum_{\bold{n}\in \mc{B}(\mbf{x})}|h(L(\bold{n}))-\mu_{h}(AX)|^2\right)^{1/2}\ll \frac{\langle \mbf{x}\rangle }{\sqrt{q}}\left(\mathbb{D}^*(f,1;AX) + \frac{1}{(\log X)^{B/2}}\right).
\end{align*}
For each $p\leq AX$, put
\[\mu_{h,p}=\left(1-\frac{1}{p}\right)\sum_{k : p^k\le AX}\frac{h(p^k)}{p^k}\]
so that in light of Remark \ref{SIMP}, $\mu_h(AX) = \sum_{p \leq AX} \mu_{h,p}$. Observe that
\begin{align}
e^{\mu_{h,p}}&=1+\mu_{h,p}+O(\mu_{h,p}^2) \nonumber\\
&=1 + \left(1-\frac{1}{p}\right)\sum_{k\geq 0}\frac{f(p^k)}{p^k}-\left(1-\frac{1}{p}\right)\sum_{k \geq 1} \frac{1}{p^k} +O\left(\left|\sum_{p^k\le AX}\frac{h(p^k)}{p^k}\right|^2 + \sum_{p^k > AX} \frac{1}{p^k}\right) \label{TAYLOR1}\\
&= \left(1-\frac{1}{p}\right)\sum_{k \geq 0} \frac{f(p^k)}{p^k} + O\left(\frac{1}{p}\sum_{p^k \leq AX} \frac{|f(p^k)-1|^2}{p^k}+(AX)^{-1} \right),\label{TAYLOR2}
\end{align}
where we applied the Cauchy-Schwarz inequality to the first error term. \\
Since $\text{Re}(h(p^k)) \leq 0$ for all $k\geq 1$, $|e^{\mu_{h,p}}| \leq 1$; also, $|\mathfrak{P}(f,AX)| \leq 1$ trivially. Thus, applying \eqref{iterate1} and the Cauchy-Schwarz inequality once again,
 \begin{align*}
 |e^{\mu_{h}(AX)}-\mathfrak{P}(f;AX)| &\leq \sum_{p\le AX}\left|e^{\mu_{h,p}}-\left(1-\frac{1}{p}\right)\sum_{k\geq 0}\frac{f(p^k)}{p^k}\right|
 \\&\ll \sum_{p^k\le AX}\frac{1}{p}\frac{|f(p^k)-1|^2}{p^k}+(AX)^{-1}\sum_{p\le AX} 1\\
&\ll \mathbb{D}^{\ast}(f,1;AX)+\frac{1}{\log (AX)}.
 \end{align*}
This implies that
\begin{equation*}
e^{\mu_h(AX)} \sum_{\mbf{n} \in \mc{B}(\mbf{x})} g(\mbf{n})1_{q|L(\mbf{n})} = \mathfrak{P}(f;AX)\sum_{\mbf{n} \in \mc{B}(\mbf{x})} g(\mbf{n})1_{q|L(\mbf{n})} + O\left(\frac{\lla \mbf{x}\rra}{q}\left(\mathbb{D}^{\ast}(f,1;AX)+\frac{1}{\log (AX)}\right)\right),
\end{equation*}
and the claim follows.
 \end{proof}
Next, we show how the factors $\mathfrak{P}(f_j;X)$ relate to the $p$-adic local factors $M_p(\mbf{f},\mbf{L})$.
\begin{lem} \label{LOCAL}
Let $X \geq y \geq 2$. Let $\mbf{L}$ be a primitive integral system of size $k$, and let $\mbf{f}$ be a vector of $k$ 1-bounded multiplicative functions that are supported on prime powers $p^{\mu} > y$. Then, as $y \ra \infty$,
\begin{equation*}
\prod_{y < p \leq X} M_p(\mbf{f};\mbf{L}) = \left(1+O_k\left(\frac{1}{\log y}\right)\right)\left( \prod_{1 \leq j \leq k} \mathfrak{P}(f_j;X)+O\left(y^{-1+o(1)}\right)\right).
\end{equation*}
\end{lem}
\begin{proof}
Let $x$ be large positive real number. We have
\begin{align}
x^{-l} \sum_{\mbf{n} \in [x]^l} \prod_{1 \leq j \leq k} f_{j,p}(L_j(\mbf{n})) &= x^{-l} \sum_{\nu_1,\ldots,\nu_k \geq 0} \prod_{1 \leq j \leq k} f_j(p^{\nu_j})\sum_{\mbf{b}^{(1)} \in (\mb{Z}/p^{\nu_1}\mb{Z})^l \atop p^{\nu_1} || L_1(\mbf{b}^{(1)})}\cdots \sum_{\mbf{b}^{(k)} \in (\mb{Z}/p^{\nu_k}\mb{Z})^l \atop p^{\nu_k} || L_k(\mbf{b}^{(k)})} \sum_{\mbf{n} \in [x]^l \atop n_j \equiv b_j^{(t)} (p^{\nu_t}) \forall j,t} 1 \nonumber\\
&= x^{-l} \sum_{\nu_1,\ldots,\nu_k \geq 0} \prod_{1 \leq j \leq k} f_j(p^{\nu_j})\sum_{\mbf{b}^{(1)} \in (\mb{Z}/p^{\nu_1}\mb{Z})^l \atop p^{\nu_1} || L_1(\mbf{b}^{(1)})}\cdots \sum_{\mbf{b}^{(k)} \in (\mb{Z}/p^{\nu_k}\mb{Z})^l \atop p^{\nu_k} || L_k(\mbf{b}^{(k)})} \prod_{1 \leq j \leq l} \sum_{n_j \leq x \atop n_j \equiv b_j^{(t)} (p^{\nu_t}) \forall t} 1. \label{QUANT}
\end{align}
By the Chinese remainder theorem, for each $j$ there is a unique solution modulo $p^{\max_{1 \leq t \leq k} \nu_t}$ to the $k$ simultaneous congruences in the inner sum of \eqref{QUANT} if, and only if, $b_j^{(r)} \equiv b_j^{(s)} (p^{\min(\nu_r,\nu_s)})$ for each $1 \leq r < s \leq k$. Hence, the right side of \eqref{QUANT} is
\begin{equation*}
\sum_{0 \leq \nu_1,\ldots,\nu_k \leq \log x/\log p} \left(\prod_{1 \leq j \leq k} f_j(p^{\nu_j}) \right)p^{-l\max_t \nu_t}\mathop{\sum_{\mbf{b}^{(1)} \in (\mb{Z}/p^{\nu_1}\mb{Z})^l \atop p^{\nu_1} || L_1(\mbf{b}^{(1)})}\cdots \sum_{\mbf{b}^{(k)} \in (\mb{Z}/p^{\nu_k}\mb{Z})^l \atop p^{\nu_k} || L_k(\mbf{b}^{(k)})}}_{b_t^{(r)} \equiv b_t^{(s)} (p^{\min(\nu_r,\nu_s)}) \forall r,s} 1 + O\left(x^{-1} \left(\frac{\log X}{\log p}\right)^l\right).
\end{equation*}
Taking $x \ra \infty$, we therefore have
\begin{equation}
\prod_{y < p \leq X} M_p(\mbf{f},\mbf{L}) = \prod_{y < p \leq X} \sum_{\nu_1,\ldots,\nu_k \geq 0} \left(\prod_{1 \leq j \leq k} f_j(p^{\nu_j}) \right)p^{-l\max_t \nu_t}\mathop{\sum_{\mbf{b}^{(1)} \in (\mb{Z}/p^{\nu_1}\mb{Z})^l \atop p^{\nu_1} || L_1(\mbf{b}^{(1)})}\cdots \sum_{\mbf{b}^{(k)} \in (\mb{Z}/p^{\nu_k}\mb{Z})^l \atop p^{\nu_k} || L_k(\mbf{b}^{(k)})}}_{b_t^{(r)} \equiv b_t^{(s)} (p^{\min(\nu_r,\nu_s)})} 1. \label{PADICS}
\end{equation}
Now, consider the product of the factors $\mathfrak{P}(f_j;X)$,
i.e., 
\begin{equation*}
\prod_{1 \leq j \leq k} \mathfrak{P}(f_j; X) = \prod_{1 \leq j \leq k} \prod_{p \leq X} \sum_{\nu_j \geq 0} f_j(p^{\nu_j}) \left(\frac{\omega_{L_j}(p^{\nu_j})}{p^{\nu_jl}} - \frac{\omega_{L_j}(p^{\nu_j+1})}{p^{(\nu_j+1)l}}\right). 
\end{equation*}
By the prime number theorem, the contribution from $p \leq y$ is
\begin{equation}\label{SMALLY}
\prod_{1 \leq j \leq k} \prod_{p \leq y} \left(1+O\left(\sum_{p^{\nu} > y \atop p \leq y} \frac{1}{p^{\nu}}\right)\right) = \prod_{1 \leq j \leq k} \left(1+O\left(y^{-1}\pi(y)\right)\right) = 1+O\left(\frac{k}{\log y}\right),
\end{equation}
whence
\begin{align*}
\prod_{1 \leq j \leq k} \mathfrak{P}(f_j;X) &= \left(1+O_k\left(\frac{1}{\log y}\right)\right)\prod_{1 \leq j \leq k} \mathfrak{P}(f_j;y,X)\\
&= \left(1+O_k\left(\frac{1}{\log y}\right)\right)\prod_{y < p \leq X} \sum_{\nu_1,\ldots,\nu_k \geq 0} \prod_{1 \leq j \leq k} f(p^{\nu_j})\left(\frac{\omega_{L_j}(p^{\nu_j})}{p^{\nu_jl}}-\frac{\omega_{L_j}(p^{\nu_j+1})}{p^{(\nu_j+1)l}}\right).
\end{align*}
Subtracting $\prod_{y < p \leq X} M_p(\mbf{f},\mbf{L})$ from $\prod_{1 \leq j \leq k} \mathfrak{P}(f_j;y,X)$ and using the fact that $|\mathfrak{P}(f_j;y,X)| \leq 1$ for each $j$, \eqref{iterate1} gives
\begin{align}
&\left|\prod_{y < p \leq X} M_p(\mbf{f},\mbf{L}) - \prod_{1 \leq j \leq k} \mathfrak{P}(X,y;f_j)\right| \nonumber\\
&\leq \sum_{y < p \leq X} \sum_{\nu_1,\ldots,\nu_k \geq 0} \left|\prod_{1 \leq j \leq k} \left(\frac{\omega_{L_j}(p^{\nu_j})}{p^{\nu_jl}}-\frac{\omega_{L_j}(p^{\nu_j+1})}{p^{(\nu_j+1)l}}\right) - p^{-l\max_t \nu_t}\mathop{\sum_{\mbf{b}^{(1)} \in (\mb{Z}/p^{\nu_1}\mb{Z})^l \atop p^{\nu_1} || L_1(\mbf{b}^{(1)})}\cdots \sum_{\mbf{b}^{(k)} \in (\mb{Z}/p^{\nu_k}\mb{Z})^l \atop p^{\nu_k} || L_k(\mbf{b}^{(k)})}}_{b_t^{(r)} \equiv b_t^{(s)} (p^{\min(\nu_r,\nu_s)})} 1\right|. \label{COMPMPMFP}
\end{align}
Observe that when at most one of the indices $1 \leq j \leq k$ satisfies $\nu_j \geq 1$, the compatibility condition on the vectors $\mbf{b}^{(t)}$ is automatically satisfied, and can hence be dropped. Thus, for $\sum_{1 \leq j \leq k} \nu_j \leq 1$, the $k$ sums over vectors $\mbf{b}^{(t)}$ in \eqref{PADICS} are precisely
\begin{equation*}
p^{-l\max_t \nu_t} \sum_{\mbf{b}^{(1)} \in (\mb{Z}/p^{\nu_1}\mb{Z})^l \atop p^{\nu_1} || L_1(\mbf{b}^(1)}\cdots \sum_{\mbf{b}^{(k)} \in (\mb{Z}/p^{\nu_k}\mb{Z})^l \atop p^{\nu_k} || L_k(\mbf{b}^{(k)}}1 = \prod_{1 \leq j \leq k} \left(\frac{\omega_{L_j}(p^{\nu_j})}{p^{\nu_jl}} - \frac{\omega_{L_j}(p^{\nu_j+1})}{p^{(\nu_j+1)l}}\right).
\end{equation*}
By well-known results on partitions (see, for instance \cite{ErdPart}), the number of terms in the $\nu_j$ sums with $\sum_{1 \leq j \leq k} \nu_j = m \geq 2$ is at most $e^{C\sqrt{m}}$, where $C > 0$ is absolute. Since each of the inner terms in \eqref{COMPMPMFP} has size $O\left(p^{-m}\right)$, it follows that
\begin{equation*}
\left|\prod_{y < p \leq X} M_p(\mbf{f},\mbf{L}) - \prod_{1 \leq j \leq k} \mathfrak{P}(f_j;y,X)\right| \ll \sum_{y < p \leq X} \sum_{m \geq 2} e^{C\sqrt{m}}p^{-m}\ll \sum_{y < p \leq X} p^{-2+o(1)}\ll y^{-1+o(1)}.
\end{equation*}
Combining this with \eqref{SMALLY} completes the proof.
\end{proof}
Lastly, we shall require the following smooth numbers estimate due to DeBruijn \cite{DeB}. Recall that for $x \geq y \geq 2$, $\Psi(x,y)$ denotes the number of integers less than or equal to $x$, all of whose prime factors are less than or equal to $y$.
\begin{lem} \label{SMOOTH}
For $x \geq y \geq 2$,
\begin{equation*}
\log \Psi(x,y) = \left(1+o(1)\right) \left(\frac{\log x}{\log y}\log\left(1+\frac{y}{\log x}\right) + \frac{y}{\log y}\log\left(1+\frac{\log x}{y}\right)\right).
\end{equation*}
\end{lem}
\subsection{Technical Results for Proposition~\ref{MultAvgNonPret}.}
As mentioned in the introduction, Proposition~\ref{MultAvgNonPret} follows from Theorem 1.6 of \cite{MRT}. A special case of the latter, which we use in the sequel (see Section 4), is as follows.
\begin{thm}[\cite{MRT}, Theorem 1.6] \label{ThmMRT}
Fix $A,m \geq 1$ and let $x \geq 10$. Let $g_1,\ldots,g_k$ be 1-bounded, complex-valued multiplicative functions and let $c_1,\ldots,c_k,b_1,\ldots,b_k \in \mb{N}$ be such that $c_j,b_j \leq A$ for each $j$. Then for each $1 \leq j_0 \leq x$,
\begin{equation*}
x^{-(k+1)}\sum_{1 \leq h_1,\ldots,h_{k-1} \leq mAx} \left|\sum_{1 \leq n \leq x} \prod_{1 \leq j \leq k} g_j(c_jn +b_j + h_j)\right| \ll m^{k-1} k^2 A^k \left(e^{-\mc{D}_{j_0}(x)/80}+(\log x)^{-1/3000} \right).
\end{equation*}
\end{thm}
\begin{rem}
Strictly speaking, in the statement of Theorem 1.6 in \cite{MRT} the range of $h_j$ is bounded above by $x$, rather than by $mAx$, as written here. However, for fixed $m$ and $A$, a careful look at the proof there shows that a perturbation of $H$ by a fixed quantity does not affect their arguments (which depend at most on $\log H$).
\end{rem}
It turns out that we can reduce the proof of Proposition~\ref{MultAvgNonPret} to showing that a similar statement holds when the system of linear forms $\mbf{L}$ is a Gowers system. This is a consequence of Lemma \ref{GVNT} below, which allows us to prove a quantitative refinement of Lemma 3.4 in \cite{HF} in Section 4. To state Lemma \ref{GVNT} precisely, we recall the following definition (see Definition 1.3.2 in \cite{HOFA}).
\begin{def1} \label{CSCOMP}
A collection $\mbf{L}$ of $k$ integral linear forms in $l$-variables on a finite Abelian group $G$ is said to have \emph{Cauchy-Schwarz complexity} at most $s$ if, for each $1 \leq j \leq k$ we can partition the set of forms $\{L_1,\ldots,L_k\}\bk \{L_j\}$ into $s+1$ classes $\{C_t: 1 \leq t \leq s+1\}$ such that $L_j \notin \text{Span}(C_t)$ for each $1 \leq t \leq s+1$. (If no such $s$ exists then the collection of forms is said to have Cauchy-Schwarz complexity $\infty$.)
\end{def1} 
Note that if $k \geq 2$ then a primitive integral system of $k$ linear forms always has Cauchy-Schwarz complexity at most $k-2$, by taking the partition of singletons. Also, if an integral system of linear forms is \emph{linearly independent} then the Cauchy-Schwarz complexity is at most $0$. \\
We may now state the following lemma, which is Exercise 1.3.23 in \cite{HOFA}.
\begin{lem}[Generalized von Neumann Inequality] \label{GVNT}
Let $G$ be a finite Abelian group and let $\psi_1,\ldots,\psi_k : G^l \ra G$ be a set of integral linear forms with Cauchy-Schwarz complexity at most $s$. If $f_1,\ldots,f_k : G \ra \mb{C}$ are 1-bounded functions on $G$ then
\begin{equation*}
|G|^{-l} \sum_{\mbf{g} \in G^l} \prod_{1 \leq j \leq k} f_j(\psi_j(\mbf{g})) \ll_{k,l} \min_{1 \leq j \leq k} \|f_j\|_{U^{s+1}(G)}.
\end{equation*}
\end{lem}
\section{Proof of Theorem~\ref{MultAvg}}
As in the statement of Theorem \ref{MultAvg} put $F_j(n) := f_j(n) \bar{\chi_j}(n) n^{-it_j}$ when $(n,q_j) = 1$, and $F_j(n) = 1$ otherwise. Furthermore, 
let $F_j = F_{j,s} \cdot F_{j,l}$, where we set 
\begin{equation*}
F_{j,s}(p^k) := \begin{cases} F_j(p^k) &: \ p^k \leq y \\ 1 &: \ p^k > y \end{cases},  \ \ \ F_{j,l}(p^k) := \begin{cases} 1 &: \ p^k \leq y \\ F_j(p^k) &: \ p^k > y. \end{cases}
\end{equation*}
Given vectors $\mbf{a},\mbf{n} \in \mb{N}^k$ let $$h_{\mbf{a}}(\mbf{n}) := \prod_{1 \leq j \leq k} (\chi_j \cdot F_{j,s})\left(L_j(\mbf{n})/a_j\right)\left(L_j(\mbf{n})/a_j\right)^{it_j}1_{a_j|L_j(\mbf{n})}$$ (otherwise, set $h_{\mbf{a}}(\mbf{n}) = 0$). Note that $h_{\mbf{a}}(\mbf{n})$ is supported on vectors $\mbf{n}$ such that $a_j|L_j(\mbf{n})$ and $(L_j(\mbf{n})/a_j,q_j) = 1$ for each $j$. Thus,
\begin{align*}
\sum_{\mbf{n} \in \mc{B}(\mbf{x})} \prod_{1 \leq j \leq k} f_j(L_j(\mbf{n})) &= \sum_{\text{rad}(a_j)|q_j \forall j} \prod_{1 \leq j \leq k} f_j(a_j)\left(\sum_{\mbf{n} \in \mc{B}(\mbf{x}) \atop (L_j(\mbf{n})/a_j,q_j) = 1} \prod_{1 \leq j \leq k} f_j\left(L_j(\mbf{n})/a_j\right)1_{a_j|L_j(\mbf{n})}\right) \\
&= \sum_{\text{rad}(a_j)|q_j \forall j} \prod_{1 \leq j \leq k} f_j(a_j)\left(\sum_{\mbf{n} \in \mc{B}(\mbf{x})}h_{\mbf{a}}(\mbf{n})\prod_{1 \leq j \leq k} F_{j,l}\left(L_j(\mbf{n})/a_j\right)1_{a_j|L_j(\mbf{n})}\right) \\
&=: \lla \mbf{x}\rra\sum_{\text{rad}(a_j)|q_j \forall j} \prod_{1 \leq j \leq k} f_j(a_j)\mc{M}_{\mbf{a}}(\mbf{x};\mbf{f},\mbf{L}). 
\end{align*}
For each $j$ and $q_j$, define $R_{q_j}(m) := \max\{d|m : \rad(d)|q_j\}$. It is easy to see that $R_{q_j}$ is multiplicative. Thus, define $F_{j,l}^{\ast}(n) := F_{j,l}\left(\frac{n}{R_{q_j}(n)}\right)$, and note that this, too, is clearly multiplicative. Moreover, we have $h_{\mbf{a}}(\mbf{n}) \neq 0$ if, and only if, $a_j = R_{q_j}(L_j(\mbf{n}))$ and hence $F_{j,l}\left(L_j(\mbf{n})/a_j\right) = F_{j,l
}^{\ast}(L_j(\mbf{n}))$ in this case. Applying Lemma \ref{key} repeatedly, we thus have
\begin{align}
&\sum_{\text{rad}(a_j)|q_j \forall j} \prod_{1 \leq j \leq k} f_j(a_j)\mc{M}_{\mbf{a}}(\mbf{x};\mbf{f},\mbf{L}) = \sum_{\text{rad}(a_j)|q_j \forall j} \prod_{1 \leq j \leq k} f_j(a_j) \mathfrak{P}(F^{\ast}_{j,l};AX)\left(\langle \mbf{x}\rangle^{-1}\sum_{\mbf{n} \in \mc{B}(\mbf{x})} h_{\mbf{a}}(\mbf{n})\right) \\
&+ O\left(\sum_{1 \leq j \leq k} \left(\sum_{\text{rad}(a_j)|q_j} \frac{1}{\sqrt{a_j}}\right)\left(\mb{D}^{\ast}(f_j,\chi_jn^{it_j}; y,AX) + \frac{k}{(\log X)^{B'}}\right)\right) \nonumber\\
&= \left(1+O\left(\frac{k}{\log y}\right)\right)\left(\prod_{y < p \leq AX} M_p\left(\mbf{F},\mbf{L}\right) + O\left(y^{-1+o(1)}\right)\right)\sum_{\text{rad}(a_j)|q_j \forall j} \prod_{1 \leq j \leq k} f_j(a_j)\left(\langle \mbf{x}\rangle^{-1}\sum_{\mbf{n} \in \mc{B}(\mbf{x})} h_{\mbf{a}}(\mbf{n})\right) \\
&+ O\left(\sum_{1 \leq j \leq k} \left(\prod_{p|q_j}\left(1-\frac{1}{\sqrt{p}}\right)^{-1}\right)\left(\mb{D}^{\ast}\left(f_j,\chi_jn^{it_j};y, AX\right)+\frac{k}{(\log X)^{B'}}\right)\right), \label{NOW}
\end{align}
where in \eqref{NOW} we used Lemma \ref{LOCAL}, coupled with the fact that $F_{j,l}^{\ast}(p^k) = F_{j,l}(p^k)$ except for the prime divisors of $q_j$ which we assume are inferior to $y$. \\
For two vectors $\mbf{a}$ and $\mbf{b}$ of the same length we will write $\mbf{a} \preceq \mbf{b}$ to mean that $a_j \leq b_j$ for each $j$. Now, for $\mbf{z} \in \mc{B}(\mbf{x})$ and $\mbf{a}$ such that $\rad(a_j)|q_j$ for each $j$, set
\begin{equation*}
G_{\mbf{a}}(\mbf{z}) := \sum_{\mbf{n} \preceq \mbf{z}} \prod_{1 \leq j \leq k} \chi_j(L_j(\mbf{n})/a_j)F_{j,s}(L_j(\mbf{n})/a_j)1_{a_j|L_j(\mbf{n})},
\end{equation*}
so that by partial summation,
\begin{equation}
\sum_{\mbf{n} \in \mc{B}(\mbf{x})} h_{\mbf{a}}(\mbf{n}) = \left(\prod_{1 \leq j \leq k} a_j^{-it_j}\right)\int_{\mc{B}(\mbf{x})} \prod_{1 \leq j \leq k} L_{j}(\mbf{u})^{it_j} dG_{\mbf{a}}(\mbf{u}). \label{STIELTJES}
\end{equation}
Note that we can express $G_{\mbf{a}}(\mbf{z})$ as
\begin{align}
G_{\mbf{a}}(\mbf{z}) &= \sideset{}{^{\ast}}\sum_{u_1 (q_1)} \cdots \sideset{}{^{\ast}}\sum_{u_k (q_k)} \left(\prod_{1 \leq j \leq k} \chi_j(u_j)\right) \sum_{\mbf{n} \preceq \mbf{z} \atop L_j(\mbf{n)}/a_j \equiv u_j (q_j)\forall j} \prod_{1 \leq j \leq k} F_{j,s}(L_j(\mbf{n})/a_j) \nonumber\\
&=: \mathop{\sideset{}{^{\ast}}\sum_{u_1 (q_1)} \cdots \sideset{}{^{\ast}}\sum_{u_k (q_k)}}_{\exists \mbf{n} : L_j(\mbf{n})/a_j \equiv u_j (q_j) \forall j} \left(\prod_{1 \leq j \leq k} \chi_j(u_j)\right) R_{\mbf{a}}(\mbf{z};\mbf{u}). \label{ZEROTH}
\end{align}	
Define $g_{j,s} := \mu \ast F_{j,s}$, and let $Y := e^{3y}$. It follows by induction on $\nu$ that $g_{j,s}(p^{\nu+1}) = 0$ whenever $p^{\nu} > y$. By the prime number theorem, $\left(\prod_{p^k \leq y} p\right)^2 = e^{(2+o(1))y} \leq Y$, and thus all divisors in the support of $g_{j,s}$ are at most $Y$ when $y$ is sufficiently large. Let $\mbf{1} := (1,\ldots,1)$. By M\"{o}bius inversion,
\begin{align}
R_{\mbf{a}}(\mbf{z};\mbf{u}) 
&= \sum_{\mbf{d} \in \mc{B}(Y\mbf{1})\atop P^+(d_j) \leq y, (d_j,q_j) = 1 \ \forall \ j} \left(\prod_{1 \leq j \leq k} g_{j,s}(d_j)\right)\sum_{\mbf{n} \preceq \mbf{z} \atop a_jd_j |L_j(\mbf{n}), L_j(\mbf{n})/a_j \equiv u_j (q_j) \ \forall j} 1. \label{FIRST}
\end{align}
Let $S_{\mbf{a},\mbf{d}}(\mbf{L};\mbf{u},\mbf{v})$ denote the set of solutions to the $2k$ simultaneous congruences $L_j(\mbf{n})/a_j \equiv u_j (q_j)$, $L_j(\mbf{n})/a_j \equiv v_j (d_j)$ for all $1 \leq j \leq k$, and let $R_{\mbf{a},\mbf{d}}(\mbf{L}; \mbf{u},\mbf{v})$ denote the density of this set. Then
\begin{align}
R_{\mbf{a}}(\mbf{z};\mbf{u}) &= \lla \mbf{z}\rra \left(\sum_{\mbf{d} \in \mc{B}(Y\mbf{1}) \atop P^+(d_j) \leq y, (d_j,q_j) = 1 \ \forall j} \prod_{1 \leq j \leq k} g_{j,s}(d_j) R_{\mbf{a},\mbf{d}}(\mbf{L};\mbf{u},\mbf{0}) + O\left(\left(\sum_{1 \leq j \leq k} z_j^{-1}\right) \frac{\Psi(Y,y)^k}{[a_1,\ldots,a_k]}\right)\right), \label{NEXTONE}
\end{align}
It is easy to see that $R_{\mbf{a},\mbf{d}}(\mbf{L};\mbf{u},\mbf{0})$, given that it is non-zero, is independent of $\mbf{u}$. Indeed, note that for any $\mbf{r},\mbf{s} \in \mb{R}^l$,
\begin{align} 
L_j(\mbf{r}-\mbf{s}) &= L_j(\mbf{r}) - L_j(\mbf{s}) + L_j(\mbf{0}), \label{AFFINE1} \\
L_j(\mbf{r}+\mbf{s}) &= L_j(\mbf{r}) + L_j(\mbf{s}) - L_j(\mbf{0}). \label{AFFINE2}
\end{align}
This implies immediately that if there exists a vector $\mbf{n}$ such that $L_j(\mbf{n})/a_j \equiv u_j (q_j)$ and $L_j(\mbf{n})/a_j \equiv 0 (d_j)$ 
then for any such $\mbf{n}$ we have $S_{\mbf{a},\mbf{d}}(\mbf{L};\mbf{u},\mbf{0}) = S_{\mbf{a},\mbf{d}}(\mbf{L}-\mbf{L}(\mbf{0});\mbf{0},\mbf{0}) + \mbf{n}$ (where, for an Abelian group $G$ and a subset $S$ of $G$, we write $S + v := \{s+v : s \in S\}$ for $v \in G$). Using this remark in \eqref{NEXTONE}, inserting the latter into \eqref{ZEROTH} and
applying the bound $|g_{j,s}(d_j)| \leq \tau(d_j) \leq 2^{\pi(y)}$ for each $j$, it follows that
\begin{align*}
G_{\mbf{a}}(\mbf{z}) &= \lla \mbf{z} \rra \Xi_{\mbf{a}}(\mbf{\chi},\mbf{L})\sum_{\mbf{d} \in \mc{B}(Y\mbf{1}) \atop P^+(d_j) \leq y, (d_j,q_j) = 1 \ \forall j} \prod_{1 \leq j \leq k} g_{j,s}(d_j) R_{\mbf{a},\mbf{d}}(\mbf{L}-\mbf{L}(\mbf{0});\mbf{0},\mbf{0}) \\
&+ O\left(\lla \mbf{z}\rra 2^{k\pi(y)}\Psi(Y,y)^k(q_1 \cdots q_k)\sum_{1 \leq j \leq l} z_j^{-1}\right).
\end{align*}
By Lemma \ref{SMOOTH} there is a constant $C \leq 9/4$ such that $\Psi(Y,y) \leq e^{Cy/\log y}$, so we may replace the error term above by
$O\left(e^{\frac{3ky}{\log y}}(q_1\cdots q_k)E(\mbf{z})\right)$, where $E(\mbf{z}) := \lla \mbf{z}\rra  \left(\sum_{1 \leq j \leq l} z_j^{-1}\right)$.
The integral in \eqref{STIELTJES} takes the shape
\begin{align}
&\Xi_{\mbf{a}}(\mbf{\chi},\mbf{L}) \sum_{\mbf{d} \in \mc{B}(Y\mbf{1}) \atop P^+(d_j) \leq y \forall j}R_{\mbf{a},\mbf{d}}(\mbf{L}-\mbf{L}(\mbf{0});\mbf{0}, \mbf{0})\prod_{1 \leq j \leq k} g_{j,s}(d_j)\int_{\mc{B}(\mbf{x})} \prod_{1 \leq j \leq k} L_j(\mbf{u})^{it_j}d\mbf{u} \nonumber\\
&+ O\left(e^{\frac{3ky}{\log y}}(q_1\cdots q_k)\left|\int_{\mc{B}(\mbf{x})} \prod_{1 \leq j \leq k} L_j(\mbf{u})^{it_j}dE(\mbf{u}) \right|\right) \nonumber\\
&=: T_1 + e^{\frac{3ky}{\log y}}(q_1\cdots q_k)T_{2}. \label{TEXPS}
\end{align}
Now, rescaling the integral in $T_1$, we have
\begin{align}
\int_{\mc{B}(\mbf{x})} \left(\prod_{1 \leq j \leq k} L_j(\mbf{u})^{it_j}\right) d\mbf{u}  &= \langle \mbf{x}\rangle  \int_{\prod_{1 \leq s \leq l} [1/x_s,1]} \left(\prod_{1 \leq j \leq k}L_j((u_1x_1,\ldots,u_lx_l))^{it_j}\right) d\mbf{u} \nonumber\\
&= \left(1+O\left(\frac{lA}{x_-}\right)\right)\langle \mbf{x}\rangle \mc{I}(\mbf{x};\mbf{L},\mbf{t}), \label{ARCH}
\end{align}
whence that 
\begin{equation*}
T_1 = \left(1+O(lAx_-^{-1})\right)\langle \mbf{x}\rangle \Xi_{\mbf{a}}(\mbf{\chi},\mbf{L}) \mc{I}(\mbf{x};\mbf{L},\mbf{t})\sum_{\mbf{d} \in \mc{B}(Y\mbf{1}) \atop P^+(d_j) \leq y, (d_j,q_j) = 1 \forall j}R_{\mbf{a},\mbf{d}}(\mbf{L}-\mbf{L}(\mbf{0});\mbf{0},\mbf{0}) \prod_{1 \leq j \leq k} g_{j,s}(d_j).
\end{equation*}
We next consider $T_{2}$ as in \eqref{TEXPS}. Applying partial summation repeatedly, we can write it as
\begin{align}
\sum_{0 \leq m \leq l} (-1)^{l-m}\sum_{1 \leq j_1 < \cdots < j_m \leq l} \int_{u_{r_s} \leq x_{r_s} \forall s \atop r_s \neq j_v \forall s,v}du_{r_1} \cdots du_{r_{l-m}} \left[E_{\mbf{d}}(\mbf{u})\left(\prod_{1 \leq s \leq l-m} \frac{\partial}{\partial u_{r_s}} \right)\prod_{1 \leq j \leq k} L_j(\mbf{u})^{it_j}\right]_{u_{j_v} = 1 \atop \forall 1 \leq v \leq m}^{x_{j_v}}. \label{IBP}
\end{align}
Observe that if $L_j$ has a non-zero $u_r$ coefficient, say $c_{j,r}$,
\begin{equation} \label{DERIVBD}
\left|\frac{\partial}{\partial u_r} L_j(\mbf{u})^{it_j}\right| \leq |t_j|c_{j,r} L_j(\mbf{u})^{-1} \leq |t_j|u_{j_r}^{-1};
\end{equation}
otherwise, the $u_r$ partial derivative of $L_j^{it_j}$ is 0. Now fix $0 \leq m \leq l-1$ and a set of indices $1 \leq j_1 < \cdots < j_m \leq l$. Since the non-zero coefficients of $L_j(\mbf{u})$ are positive integers, taking further derivatives as in \eqref{DERIVBD} gives
\begin{align*}
&\left|\int_{u_{r_s} \leq x_{r_s} \forall s \atop r_s \neq j_v \forall s,v}du_{r_1} \cdots du_{r_{l-m}} \left[E_{\mbf{d}}(\mbf{u})\left(\prod_{1 \leq s \leq l-m} \frac{\partial}{\partial u_{r_s}} \right)\prod_{1 \leq j \leq k} L_j(\mbf{u})^{it_j}\right]_{u_{j_v} = 1\atop \forall 1 \leq v \leq m}^{x_{j_v}}\right| \\
&\ll_l \prod_{1 \leq j \leq k} \max\{1,|t_j|\}\int_{u_{r_s} \leq x_{r_s} \forall s \atop r_s \neq j_v \forall s,v}du_{r_1} \cdots du_{r_{l-m}}  \left[\left|E_{\mbf{d}}(\mbf{u})\right|\left(\prod_{1 \leq s \leq l-m} u_{r_s}^{-1}\right)\right]_{u_{j_v} = 1 \atop \forall 1 \leq v \leq m}^{x_{j_v}}.
\end{align*}
By the definition of $E_{\mbf{d}}(\mbf{z})$,
\begin{align*}
&\int_{u_{r_s} \leq x_{r_s} \forall s \atop r_s \neq j_v \forall s,v}du_{r_1} \cdots du_{r_{l-m}} \left[ \left|E_{\mbf{d}}(\mbf{u})\right|\left(\prod_{1 \leq s \leq l-m} u_{r_s}^{-1}\right)\right]_{u_{j_v} = 1 \atop \forall 1 \leq v \leq m}^{x_{j_v}}\\
 &\ll_m x_-^{-1}\left(\prod_{1 \leq v \leq m} x_{j_v}\right) \prod_{1 \leq s \leq l-m}\int_1^{x_{j_s}} \frac{du_{j_s}}{u_{j_s}} 
\ll_m x_-^{-1}\left(\prod_{1 \leq v \leq m} x_{j_v}\right)\prod_{1 \leq s \leq l-m} \left(\log x_{r_s}\right).
\end{align*}
These contributions are all smaller than the term with $m = l$, which is bounded by $\ll |E_{\mbf{d}}(\mbf{x})| \ll \lla \mbf{x} \rra x_-^{-1}$. Thus, summing over all $m$-tuples of distinct indices $j_v$ and all $m$, we get
\begin{equation*}
T_2 \ll_l\frac{\langle \mbf{x} \rangle}{x_-} \prod_{1 \leq j \leq k} \max\{1,|t_j|\}.
\end{equation*}
%
%
Thus, \eqref{STIELTJES} gives
\begin{align*}
\langle \mbf{x}\rangle^{-1} \sum_{\mbf{n} \in \mc{B}(\mbf{x})} h_{\mbf{a}}(\mbf{n}) &= \left(1+O\left(\frac{lA}{x_-}\right)\right)C_{\mbf{a}}(\mbf{x}; \mbf{L},\mbf{\chi}, \mbf{t})\left(\prod_{1 \leq j \leq k} a_j^{-it_j}\right) \mc{S}_{\mbf{a}}(Y,y;\mbf{f},\mbf{L}) \\
&+ O_{l}\left(\frac{1}{x_-}\frac{e^{\frac{3ky}{\log y}}}{[a_1,\ldots,a_k]}\prod_{1 \leq j \leq k} q_j\max\{1,|t_j|\}\right),
\end{align*}
where we put
\begin{align*}
\mc{S}_{\mbf{a}}(Y,y;\mbf{f},\mbf{L}) &:= \sum_{\mbf{d} \in \mc{B}_k(X\mbf{1}) \atop P^+(d_j) \leq y, (d_j,q_j) =1}R_{\mbf{a},\mbf{d}}(\mbf{L}-\mbf{L}(\mbf{0});\mbf{0}, \mbf{0})\prod_{1 \leq j \leq k} g_{j,s}(d_j).
\end{align*}
This coupled with \eqref{NOW} yields
\begin{align}
&\langle \mbf{x}\rangle^{-1} \sum_{\mbf{n} \in \mc{B}(\mbf{x})} \prod_{1 \leq j \leq k} f_j(L_j(\mbf{n})) \nonumber\\
&= \left(1+O_{k,l}\left(\frac{1}{\log y}\right)\right)\left(\sum_{\rad(a_j)|q_j \atop \forall 1 \leq j \leq k}\prod_{1 \leq j \leq k}\frac{f(a_j)}{a_j^{it_j}} C_{\mbf{a}}(\mbf{x};\mbf{L},\mbf{\chi},\mbf{t})\mc{S}_{\mbf{a}}(X,y; \mbf{f},\mbf{L})\right) \nonumber\\
&\cdot \left(\prod_{y < p \leq X} M_p(\mbf{F},\mbf{L})+ O\left(y^{-1+o(1)}\right)\right) +O(\mc{R}),\label{ALMOST}
\end{align}
where we have put
\begin{align*}
\mc{R} &:= \sum_{1 \leq j \leq k} \prod_{p|q_j} \left(1-\frac{1}{\sqrt{p}}\right)^{-1}\left(\mb{D}^{\ast}(f,\chi_jn^{it_j};y,AX) +\frac{1}{(\log X)^{B'}}\right) \\
&+ \frac{1}{x_-}\left(A + e^{\frac{3ky}{\log y}}(q_1 \cdots q_k)\left(\sum_{\text{rad}(a_j)|q_j \atop \forall 1 \leq j \leq k} [a_1,\ldots,a_k]^{-1}\right)\prod_{1 \leq j \leq k} \max\{1,|t_j|\}\right).
\end{align*}
We next apply Rankin's trick with $\delta = 1/2$ to show that
\begin{align*}
&\left|\left(\sum_{P^+(d_j) \leq y \atop (d_j,q_j) = 1 \forall j} - \sum_{\mbf{d} \in \mc{B}(Y\mbf{1}) \atop P^+(d_j) \leq y, (d_j,q_j) = 1\forall j}\right)R_{\mbf{a},\mbf{d}}(\mbf{L}-\mbf{L}(\mbf{0});\mbf{0}, \mbf{0})\prod_{1 \leq j \leq k} g_{j,s}(d_j) \right|\\
&\ll_k \sum_{d > Y \atop P^+(d) \leq y} \frac{\tau(d)}{d} \ll Y^{-\delta} \prod_{p \leq y} \left(1+\frac{2}{p^{1-\delta}}\right) \ll Y^{-\delta} \exp\left(2\sum_{p \leq y} p^{-1+\delta}\right) \\
&\ll e^{-(3\delta y - 2y^{\delta} \log_2 y)} \ll e^{-y}.
\end{align*}
Thus, we have
\begin{equation*}
S_{\mbf{a}}(Y,y;\mbf{f},\mbf{L}) = \sum_{P^+(d_j) \leq y \atop (d_j,q_j) = 1 \forall j} R_{\mbf{a},\mbf{d}}(\mbf{L}-\mbf{L}(\mbf{0});\mbf{0}, \mbf{0})\prod_{1 \leq j \leq k} g_{j,s}(d_j) + O_k\left(e^{-y}\right).
\end{equation*}
Moreover, replacing $g_{j,s}$ by $g_j = \mu \ast F_{j}$ here produces an error
\begin{align*}
&\left|\mc{S}_{\mbf{a}}(y;\mbf{f},\mbf{L})-\sum_{P^+(d_j) \leq y \atop (d_j,q_j) = 1 \forall j} R_{\mbf{a},\mbf{d}}(\mbf{L}-\mbf{L}(\mbf{0});\mbf{0}, \mbf{0})\prod_{1 \leq j \leq k} g_{j,s}(d_j)\right| \\
&\ll_k \sum_{P^+(d) \leq y \atop \exists p^{\nu} || d, \ p^{\nu} > y, \nu \geq 2}\frac{\tau(d)}{d} \ll \sum_{p^{\nu} > y \atop \nu \geq 2} \frac{1}{p^{\nu}} \sum_{P^+(d) \leq y} \frac{\tau(d)}{d} \\
&\ll y^{-\frac{1}{2}}\prod_{p \leq y} \left(1 + \frac{2}{p}\right) \ll \frac{(\log y)^2}{\sqrt{y}}.
\end{align*}
Thus, we have
\begin{equation*}
\mc{S}_{\mbf{a}}(Y,y;\mbf{f},\mbf{L}) 
= \mc{S}_{\mbf{a}}(y;\mbf{f},\mbf{L}) + O_k\left(\frac{(\log y)^2}{\sqrt{y}}\right),
\end{equation*}
which, combined with \eqref{ALMOST} completes the proof of Theorem \ref{MultAvg} in the general case. \\
Suppose now that $q_j = q$ for all $j$. We note first that by a simple calculation as in Lemma \ref{LOCAL},
\begin{equation*}
\prod_{p \leq y \atop p\nmid q} M_p(\mbf{f},\mbf{L}) = \sum_{P^+(d_j) \leq y \atop (d_j,q) = 1} R(d_1,\ldots,d_k)\prod_{1 \leq j \leq k} g_{j,s}(d_j),
\end{equation*}
where $R(d_1,\ldots,d_k)$ is the density of solutions in $\mb{N}^l$ to the simultaneous conditions $d_j|L_j(\mbf{n})$ for each $j$. Arguing as in the remarks surrounding \eqref{AFFINE1} and \eqref{AFFINE2}, $R(d_1,\ldots,d_k)$ is also the density corresponding to the shifted forms $L_j-L_j(\mbf{0})$. 
Now since $(q,d_j) = 1$ for all $j$, 
\begin{equation} \label{RDENSMULT}
R_{\mbf{a},\mbf{d}}(\mbf{L}-\mbf{L}(\mbf{0}),\mbf{0},\mbf{0}) = R([qa_1,a_1d_1],\cdots, [qa_k,a_kd_k]) = R(qa_1,\ldots,qa_k) R(d_1,\ldots,d_k)
\end{equation}
by multiplicativity. We thus have 
\begin{equation*}
S_{\mbf{a}}(y;\mbf{f},\mbf{L}) = R(qa_1,\cdots qa_k)\prod_{p \leq y \atop p \nmid q} M_p(\mbf{f},\mbf{L})
\end{equation*} 
whenever $\mbf{a}$ with $\rad(a_j) |q_j$ for each $j$, and Theorem \ref{MultAvg} follows as well in the special case $q_j = q$ for all $j$.
\section{Proof of Proposition~\ref{MultAvg}}
As mentioned in Section 2, we shall first make the following reduction, which is based on ideas of Green and Tao (see Theorem 7.1' and Appendix A of \cite{GT}).  For convenience, we write $\mb{Z}_N$ to mean $\mb{Z}/N\mb{Z}$.
\begin{lem} \label{REDUCE}
Let $A,k,l \geq 1$. Let $\mbf{L}$ be a primitive integral system of $k$ linear forms in $l$ variables and height at most $A$. Suppose that $f_1,\ldots,f_k: \mb{N} \ra \mb{C}$ are 1-bounded arithmetic functions such that $\min_{1 \leq j \leq k} \|f_j\|_{U^{k-1}(x)} \ra 0$ as $x \ra \infty$. Then 
\begin{equation*}
M(x;\mbf{f},\mbf{L}) \ll_{k,l,A} \min_{1 \leq j \leq k} \|f_j\|_{U^{k-1}(x)}^{\frac{1}{2}}.
\end{equation*}
Moreover, if $\mbf{L}$ is a system of \emph{linearly independent} forms then we can replace the $U^{k-1}$ norm on the right side by the $U^2$ norm.
\end{lem}
\begin{proof}
Let $\rho > \rho' > lA$ and let $N$ be a large prime satisfying $\rho'x < N \leq \rho x$, with $\rho$ sufficiently large in terms of $\rho'$ (but bounded as $x \ra \infty$). Then
\begin{equation*}
M(x;\mbf{f}, \mbf{L}) = \left(\frac{N}{x}\right)^l N^{-l} \sum_{\mbf{n} \in \mb{Z}_N^l} \prod_{1 \leq j \leq k} f_j(L_j(\mbf{n})) 1_{[1,x]^l}(\mbf{n}).
\end{equation*}
We seek to apply Lemma \ref{GVNT}, and must hence remove the weight $1_{[1,x]^l}$. To accomplish this, we use the following harmonic analytic argument, due to Green and Tao (see Proposition 7.1' of \cite{GT}). Define a metric on $\mb{Z}_N^l$ by $$d(\mbf{m},\mbf{m}') := \left(\sum_{j \leq l} \left|\frac{m_j-m_j'}{N}\right|^2\right)^{\frac{1}{2}}.$$ Let $z,Z,\lambda > 0$ be parameters to be chosen. Let $\phi_N : \mb{Z}_N^l \ra \mb{C}$ be a bounded (independently of $N$) $d$-Lipschitz map with Lipschitz constant $\lambda$ such that $\|1_{[1,x]^l} - \phi_N\|_{L^{1}(\mb{Z}^l_N)} \ll_l N^l/Z$. It is shown in Corollary A.3 of \cite{GT} that $\lambda \ll Z/N$. Expanding $\phi_N$ as a Fourier series and convolving it with the $l$-dimensional F\'{e}jer kernel of length $z$, one can show that
\begin{equation*}
\phi_N(\mbf{n}) = \sum_{\mbf{m} \in \mb{Z}_N^l} a_{\mbf{m}} e\left(\frac{\mbf{m} \cdot \mbf{n}}{N}\right) = \sum_{\mbf{m} \in [z]^l} a_{\mbf{m}}' e\left(\frac{\mbf{m} \cdot \mbf{n}}{N}\right) + O_l\left(N^l \lambda\frac{\log(z+1)}{z}\right),
\end{equation*}
where $|a_{\mbf{m}}'| \ll 1$. Inserting this expansion into our expression for $M(x;\mbf{f},\mbf{L})$, splitting the two contributions and bounding the main term trivially gives
\begin{align*}
M(x;\mbf{f},\mbf{L}) &\leq \rho^l \left(N^{-l}\left|\sum_{\mbf{n} \in \mb{Z}_n^l} \phi_N(\mbf{n}) \prod_{1 \leq j \leq k} f_j(L_j(\mbf{n}))\right| + Z^{-1}\right) \\
&\ll_{\rho,l} N^{-l}\left|\sum_{\mbf{m} \in [z]^l} a_{\mbf{m}}\sum_{\mbf{n} \in \mb{Z}_n^l}e\left(\frac{\mbf{m} \cdot \mbf{n}}{N}\right)\prod_{1 \leq j \leq k} f_j(L_j(\mbf{n}))\right| + \lambda\frac{\log(z+1)}{z} + Z^{-1} \\
&\ll_{\rho,l} \left(\sum_{\mbf{m} \in [z]^l} |a_{\mbf{m}}|\right)\max_{\mbf{m} \in \mb{Z}_N^l} N^{-l}\left|\sum_{\mbf{n} \in \mb{Z}_N^l} e\left(\frac{\mbf{m}\cdot \mbf{n}}{N}\right) \prod_{1 \leq j \leq k} f_j(L_j(\mbf{n}))\right| + \lambda\frac{\log(z+1)}{z} + Z^{-1}\\
&\ll_{\rho,l} z^l\left(N^{-l} \left|\sum_{\mbf{n} \in \mb{Z}_N^l} \prod_{0 \leq j \leq k} f_j(L_j(\mbf{n}))\right|\right) + \lambda\frac{\log(z+1)}{z} + Z^{-1},
\end{align*}
where, letting $\mbf{m}_0$ be the index maximizing the multilinear average, we let $L_0(\mbf{n}) = \mbf{m_0} \cdot \mbf{n}$ and $f_0(n) := e\left(\frac{n}{N}\right)$. Now if $L_0 \notin \text{Span}_{\mb{Q}}\{L_1,\ldots,L_k\}$ then $\{L_0,\ldots,L_k\}$ still has Cauchy-Schwarz complexity $k-2$ so by Lemma \ref{GVNT},
\begin{equation} \label{GVNTAPP}
M(x;\mbf{f},\mbf{L}) \ll_{\rho,l} z^l\min_{1 \leq j \leq k} \|f_j\|_{U^{k-1}(\mb{Z}_N)} + \lambda\frac{\log(z+1)}{z} + Z^{-1}.
\end{equation}
On the other hand, if $L_0 = \sum_{1 \leq j \leq k} \alpha_j L_j$ with $\alpha_j \in \mb{Q}$ then \eqref{GVNTAPP} still holds with $f_j'(n) := f_j(n)e(\alpha_j n)$ in place of $f_j$. Since the $U^{k-1}(\mb{Z}_N)$ norm is invariant under multiplication by exponential phases (see (B.4) in \cite{GT}) we have $\|f_j'\|_{U^{k-1}(\mb{Z}_N)} = \|f_j\|_{U^{k-1}(\mb{Z}_N)}$, and thus as written \eqref{GVNTAPP} holds in this case as well. \\
By definition, $\|f_j\|_{U^{k-1}(\mb{Z}_N)} = \|f_j\|_{U^{k-1}(x)}\|1_{[1,x]}\|_{U^{k-1}(\mb{Z}_N)} \ll_{\rho,l} \|f_j\|_{U^{k-1}(x)}$. Hence
\begin{align*} 
M(x;\mbf{f},\mbf{L}) &\ll_{\rho,l} z^l\min_{1 \leq j \leq k} \|f_j\|_{U^{k-1}(x)} + \lambda \frac{\log(z+1)}{z} + Z^{-1} \\
&\leq z^l\|f_{j_0}\|_{U^{k-1}(x)} + \lambda\frac{\log(z+1)}{z} + Z^{-1}.
\end{align*}
Suppose that $1 \leq j_0 \leq k$ is the index of the function with minimal $U^{(k-1)}(x)$ norm as $x \ra \infty$. Taking $z := \|f_{j_0}\|_{U^{k-1}(x)}^{-\frac{1}{2l}}$ and $Z = N^{1/2}$ suffices to prove the first claim. \\
The second claim follows immediately from the fact that mutually linearly independent forms have Cauchy-Schwarz complexity at most 1 trivially.
\end{proof}
\begin{lem} \label{U2}
Suppose $f$ is a 1-bounded multiplicative function such that $\mc{D}(x) := \mc{D}(f;10x,(\log x)^{1/125}) \ra \infty$ as $x \ra \infty$. Then
for some absolute $c_1,c_2 > 0$, $$\|f\|_{U^2(x)} \ll e^{-c_1\mc{D}(x)} + (\log x)^{-c_2}.$$
\end{lem}
\begin{proof}
We have
\begin{align*}
\|f\|_{U^2(x)}^4 &= x^{-3}\sum_{1 \leq n_1,n_2,n_3 \leq x} f(n_1)\bar{f(n_1+n_2)f(n_1+n_3)}f(n_1+n_2+n_3) \\
&\leq x^{-3} \sum_{1 \leq n_1,n_2 \leq x} \left|\sum_{n_3 \leq x} \bar{f(n_1+n_3)}f(n_1+n_2+n_3)\right| \\
&\leq x^{-3}\sum_{1 \leq h_1,h_2 \leq 2x} \left|\sum_{n \leq x} \bar{f(n+h_1)}f(n+h_2)\right|,
\end{align*}
upon making the change of variables $n = n_1$, $h_1 = n_3$ and $h_2 = h_1 + n_2 \leq 2x$. Applying Theorem \ref{ThmMRT} with $H = 2x$ gives
\begin{equation*}
\|f\|_{U^2(x)}^4 \ll_{k,l,A} e^{-c_1\mc{D}_{j_0}(x)} + (\log x)^{-c_2},
\end{equation*}
and the claim follows with constants $c_1/4$ and $c_2/4$ in place of $c_1,c_2$.
\end{proof}
\begin{proof}[Proof of Proposition~\ref{MultAvgNonPret}]
Proposition~\ref{MultAvgNonPret} follows immediately upon combining Lemmata \ref{REDUCE} and \ref{U2}.
\end{proof}

\section{Sign Changes of Non-Pretentious Multiplicative Functions on 3- and 4-term Arithmetic Progressions}
In this section, we study the frequency with which a given multiplicative function $f: \mb{N} \ra \{-1,1\}$ yields a given sign pattern on 3- and 4-term arithmetic progressions. 
\subsection{Sign Patterns of Non-Pretentious Functions on Fixed 3- and 4-term APs}
In order to show that certain sign patterns exhibit positive upper density, it suffices to find a corresponding lower bound for the upper logarithmic density. Our method to do this relies on the remarkable result of Tao \cite{Tao} that establishes a logarithmically averaged version of Elliott's conjecture. A special case of his result is the following.
\begin{thm}[\cite{Tao}, Corollary 1.5]
Let $b_1,b_2$ be distinct, non-negative integers. Let $f_1,f_2: \mb{N} \ra \mb{C}$ be a 1-bounded multiplicative function such that for some $j \in \{1,2\}$, $\mc{D}(f_j;Ax,\infty) \ra \infty$ as $x \ra \infty$ for each $A \geq 1$. Then
\begin{equation*}
\sum_{n \leq x} \frac{f_1(n+b_1)f_2(n+b_2)}{n} = o(\log x).
\end{equation*}
\end{thm}
We shall take advantage of this result and the unimodularity of $f$ to establish statements about correlations of $f$ with three or four translates of itself. We use the following basic device to this end.
\begin{lem} \label{BASIC}
For $n \geq 1$ let $a_1,\ldots,a_n,b_1,\ldots,b_n \in \mb{C}$ have norm uniformly bounded above by $X$. Let $w_1,\ldots,w_n\in (0,\infty)$ and put $H := \sum_{1 \leq j \leq n} w_j$. Let $A := H^{-1}\left|\sum_{1 \leq j \leq n} w_ja_j\right|$ and $B := H^{-1}\left|\sum_{1 \leq j \leq n} w_jb_j\right|$. Then
\begin{equation*}
\text{Re}\left(\sum_{1 \leq j \leq n} w_ja_j\bar{b_j}\right) \geq \left(\frac{1}{2}(A+B)^2 - X\right)H.
\end{equation*}
\end{lem}
\begin{proof}
Rotating the sums $\sum_{1 \leq j \leq n} w_ja_j$ and $\sum_{1 \leq j \leq n} w_jb_j$, we may assume without loss of generality that they point in the same direction, say $e(\theta)$. As such, by Cauchy-Schwarz,
\begin{align*}
\text{Re}\left(\sum_{1 \leq j \leq n} w_ja_j\bar{b_j}\right) &\geq \frac{1}{2}\sum_{1 \leq j \leq n} w_j\left(|a_j+b_j|^2 -2X\right) \geq \frac{1}{2H}\left(\left|\sum_{1 \leq j \leq n} w_j(a_j+b_j)\right|^2 - 2XH\right) \\
&= H\left(\frac{1}{2}\left|(A+B)e\left(\theta\right)\right|^2 -X\right),
\end{align*}
as claimed.
\end{proof}
A consequence of Lemma \ref{BASIC} is the following, which gives us a criterion to determine whether or not a multiplicative function is pretentious based on its 4-term correlations.
\begin{lem} \label{RED}
Let $x,d \geq 1$, with $d \in \mb{N}$ and $d = o(x)$. Let $f$ be a unimodular multiplicative function, and let $\delta > 0$. If $$\frac{1}{\log x} \sum_{n \leq x} \frac{f(n)f(n+d)f(n+2d)f(n+3d)}{n} > \frac{1}{\sqrt{2}}+\delta$$ as $x \ra \infty$ then there is a primitive Dirichlet character $\chi$ of conductor $q$ and a real number $t \in \mb{R}$ such that $\mb{D}(f,\chi n^{it};\infty) < \infty$.
\end{lem}
\begin{proof}
We apply Lemma \ref{BASIC} with $w_n := 1/n$, $a_n := f(n)f(n+d)f(n+2d)f(n+3d)$ and $b_n := \bar{a_{n+d}}$ for each $n \leq x$. Clearly, $a_n\bar{b_n} = f(n)\bar{f(n+4d)}$, and as $d= o(x)$,
\begin{equation*}
A+B = (2+o(1))\frac{1}{\log x}\left|\sum_{n \leq x} \frac{f(n)f(n+d)f(n+2d)f(n+3d)}{n}\right|.
\end{equation*}
By Lemma \ref{BASIC},
\begin{equation*}
\text{Re}\left(\sum_{n \leq x} \frac{f(n)\bar{f(n+4d)}}{n}\right) \geq (2+o(1))\left|\sum_{n \leq x} \frac{f(n)f(n+d)f(n+2d)f(n+3d)}{n}\right|^2 - \log x +O(1).
\end{equation*}
By assumption, it follows that 
\begin{equation*}
\text{Re}\left(\sum_{n \leq x} \frac{f(n)\bar{f(n+4d)}}{n}\right) \gg_{\delta} \log x.
\end{equation*}
The conclusion now follows from Theorem 5.1 with $f_1 = f$, $f_2 = \bar{f}$.
\end{proof}
Our next lemma is a trivial observation showing that the cardinality of the set of $n \leq x$ yielding a fixed sign pattern of a given length can be expressed as a correlation of multiplicative functions.
\begin{lem} \label{TRIVIAL}
Let $l \geq 1$, and let $\mbf{\e} \in \{-1,1\}^l$. Let $f : \mb{N} \ra \{-1,1\}$ and $g : (0,\infty) \ra \mb{R}$, and put $$S_{\mbf{\e}} := \{n \in \mb{N} : f(n+jd) = \e_j \text{ for all }0 \leq j \leq l-1\}.$$ Then 
\begin{equation*}
\sum_{n \leq x \atop n \in S_{\mbf{\e}}} g(n) = 2^{-l} \sum_{n \leq x} g(n) \prod_{0 \leq j \leq l-1} \left(1+\e_j f(n+jd)\right).
\end{equation*}
\end{lem}
\begin{proof}
If $n \notin S_{\mbf{\e}}$ then for some $0 \leq j \leq l-1$, $1+\e_jf(n+jd) = 0$, so such terms contribute nothing. Conversely, when $n \in S_{\mbf{\e}}$ then $1+\e_jf(n+jd) = 2$ for all $0 \leq j \leq l-1$, and the product is then $2^l$. This implies the claim.
\end{proof}
With these results in hand, we will establish Theorem \ref{LOGDENSSPEC}. 
\begin{proof}[Proof of Theorem \ref{LOGDENSSPEC}]
We will only prove ii). By a similar argument one can establish i) as well, and we leave the details of this to the reader. \\
Let $S_{\pm \mbf{\e}} := S_{\mbf{\e}} \cup S_{-\mbf{\e}}$. 
Write $L_{\pm \mbf{\e}}(x) := \sum_{n \leq x \atop n \in S_{\pm \mbf{\e}}} \frac{1}{n}$. Applying Lemma \ref{TRIVIAL} twice with $g(n) := \frac{1}{n}$ for all $n \in \mb{N}$,
\begin{align*}
L_{\pm \mbf{\e}}(x) &= \frac{1}{16} \sum_{n \leq x}\frac{1}{n}\left(\prod_{0 \leq j \leq 3}\left(1+\e_jf(n+jd)\right) + \prod_{0 \leq j \leq 3}\left(1-\e_jf(n+jd)\right)\right) \\
&= \frac{1}{16}\left(\sum_{S \subseteq \{0,1,2,3\}} \left(1+(-1)^{|S|}\right)\sum_{n \leq x} \frac{1}{n}\prod_{j \in S} \e_j f(n+jd)\right)\\
&= \frac{1}{8}\left(\log x + \sum_{S \subseteq \{0,1,2,3\} \atop |S| = 2} \sum_{n \leq x} \frac{1}{n}\prod_{j \in S} \e_j f(n+jd) + \e_0\e_1\e_2\e_3\sum_{n \leq x} \frac{\prod_{0 \leq j \leq 3}f(n+jd)}{n}\right).
\end{align*}
By Theorem 1.3 of \cite{Tao}, each of the six 2-element subsets $S$ of $\{0,1,2,3\}$ gives rise to $$\sum_{n \leq x} \frac{1}{n}\prod_{j \in S} f(n+jd) = o(\log x).$$ Also, by Lemma \ref{RED}, we must have
\begin{equation*}
\liminf_{x \ra \infty} \left|\frac{1}{\log x}\sum_{n \leq x} \frac{f(n)f(n+d)f(n+2d)f(n+3d)}{n}\right|\leq \frac{1}{\sqrt{2}}.
\end{equation*}
As such, we have
\begin{align*}
\limsup_{x \ra \infty} \frac{L_{\pm \mbf{\e}}(x)}{\log x}  &\geq \frac{1}{8}\left(1-\liminf_{x \ra \infty} \left|\frac{1}{\log x} \sum_{n \leq x} \frac{f(n)f(n+d)f(n+2d)f(n+3d)}{n}\right|\right) \\
&\geq \frac{1}{8}-\frac{1}{8\sqrt{2}}> \frac{1}{28}.
\end{align*}
This establishes the claim. 
\end{proof}
By a similar argument, we can show the following. The details are left to the reader.
\begin{prop}
For any $d \geq 1$, any non-pretentious function $f : \mb{N} \ra \{-1,1\}$ and any $\mbf{\e} \in \{-1,1\}^3$, the upper logarithmic density of the set of $n$ such that $f(n+jd) = \e_j$ for $j = 0,1$ and $2$ is at least $\frac{1}{28}.$
\end{prop}
\begin{proof}[Proof of Corollary \ref{GAPS4AP}]
Given a sign pattern $\mbf{\e} \in \{-1,1\}^4$ and $d \in \mb{N}$, put $$S_{\mbf{\e}}(d) := \{n \in \mb{N} : (f(n),f(n+d),f(n+2d),f(n+3d)) = \mbf{\e}\}.$$ By Proposition \ref{LOGDENSSPEC}, at least one of $S_{\mbf{\e}}(1)$ or $S_{\mbf{\e}}(1)$ has positive upper density. Without loss of generality, suppose $S_{\mbf{\e}}(1)$ has positive upper density. Thus, clearly $d(\mbf{e}) = 1$. Now, if $n \in S_{\mbf{\e}}(1)$ then $p_0n \in S_{\mbf{-\e}}(p_0)$, since $f(p_0n + p_0k) = f(p_0)f(n+k) = -\e_k$ for each $k \in \{0,1,2,3\}$. Thus, $d(-\mbf{\e}) \leq p_0$. Since $\mbf{\e}$ was arbitrary, this implies the claim.
\end{proof}
\subsection{Sign Patterns of Non-Pretentious Functions in Almost All 3-term APs}
By Chebyshev's inequality, in order to prove Theorem \ref{EQUIDIST} it suffices to show the following variance estimate.
\begin{prop} \label{VARCALCNP}
Let $\mbf{\e} \in \{-1,1\}^3$ and let $f : \mb{N} \ra \{-1,1\}$. Then
\begin{equation} \label{VARNP}
x^{-2}\sum_{d \leq x}\left(\left|\{n \leq x : f(n+jd) = \e_j \forall j\}\right| - \frac{1}{8}x\right)^2  \ll \mc{R}_f(x),
\end{equation}
where $\mc{R}_f(x)$ is as defined in \eqref{RFX}.
\end{prop}
The first and second moment calculations are given in the following two lemmata.
\begin{lem}\label{MOMENT1NP}
Let $\mbf{\e} \in \{-1,1\}^3$ and let $f : \mb{N} \ra \{-1,1\}$. Then 
\begin{equation*}
x^{-1}\sum_{d \leq x}\left|\{n \leq x : f(n+jd) = \e_j \forall j\}\right| = \frac{1}{8}x + O\left(x\mc{R}_f(x)\right).
\end{equation*}
\end{lem}
\begin{proof}
As before we have
\begin{align}
&x^{-1}\sum_{d \leq x} \left|\{n \leq x : f(n+jd) = \e_j \forall \ 0 \leq j \leq 2\}\right| = \frac{1}{8}\sum_{d,n \leq x} \prod_{0 \leq j \leq 2} \left(1+\e_jf(n+jd)\right) \nonumber\\
&= \frac{1}{8}\left(x + \sum_{S \subseteq \{0,1,2\} \atop S \neq \emptyset} \left(\prod_{j \in S} \e_j\right) \sum_{d,n \leq x} \prod_{j \in S} f(n+jd)\right). \label{SETS}
\end{align}
Fix a non-empty subset $S$ of $\{0,1,2,3\}$. The collection of forms $$\{(n,d) \mapsto n+jd : j \in S\}$$ has Cauchy-Schwarz complexity at most that of the 3-term AP $\{n + jd : 0 \leq j \leq 2\}$, which is  1. By Theorem \ref{MultAvgNonPret}, we have
\begin{equation*}
x^{-2}\sum_{d,n \leq x} \prod_{j \in S} f(n+jd) \ll \mc{R}_f(x).
\end{equation*}
As such, \eqref{SETS} can be transformed as
\begin{equation*}
x^{-1}\sum_{d \leq x}\left|\{n \leq x : f(n+jd) = \e_j \text{ for all } 0 \leq j \leq 2\}\right| = \frac{1}{8}x + O\left(x\mc{R}_f(x)\right).
\end{equation*}
\end{proof}
\begin{lem}\label{MOMENT2NP}
Let $\mbf{\e} \in \{-1,1\}^3$ and let $f : \mb{N} \ra \{-1,1\}$. Then 
\begin{equation*}
x^{-1}\sum_{d \leq x}\left|\{n \leq x : f(n+jd) = \e_j \forall j\}\right|^2 = \frac{x^2}{64} + O\left(x^2\mc{R}_f(x)\right).
\end{equation*}
\end{lem}
\begin{proof}
By Lemma \ref{TRIVIAL},
\begin{align}
&\sum_{d \leq x} \left|\{n \leq x : f(n+jd) = \e_j \forall j\}\right|^2 = \frac{1}{64} \sum_{n,n',d \leq x} \prod_{0 \leq j,j' \leq 2} \left(1+\e_jf(n+jd)\right)\left(1+\e_jf(n'+jd)\right) \nonumber\\
&= \frac{1}{64}\left(x^3 + \sum_{S,S' \subseteq \{0,1,2\} \atop S \cup S' \neq \emptyset} \left(\prod_{j \in S}\e_j \right)\left(\prod_{j' \in S'} \e_{j'}\right)\sum_{n,n',d \leq x} \left(\prod_{j \in S} f(n+jd)\right)\left(\prod_{j \in S'} f(n'+j'd)\right)\right). \label{SUMS}
\end{align}
Associate to each pair of sets $S,S' \subseteq \{0,1,2\}$ with $S\cup S' \neq \emptyset$ the system of forms $$\mbf{L}_{S,S'} := \{(n,n',d) \mapsto n+jd : j \in S\} \cup \{(n,n',d) \mapsto n'+j'd : j' \in S'\}.$$ We note to that each of the sets of forms  $\{n,n+d,n'+d\}$ and $\{n',n'+2d,n+2d\}$ (in the variables $n,n'$ and $d$) is linearly independent. This implies that the set of forms $\{n,n',n+d,n'+d,n+2d,n'+2d\}$ and each of its subsets has Cauchy-Schwarz complexity at most 1. Applying Proposition \ref{MultAvgNonPret} to $\mbf{L}_{S,S'}$, we get
\begin{equation*}
M(x;f\mbf{1},\mbf{L}_{S,S'}) \ll \mc{R}_f(x).
\end{equation*}
Thus, the second term in the brackets in \eqref{SUMS} can be bounded as
\begin{equation*}
\sum_{S,S' \subseteq \{0,1,2\} \atop S \cup S' \neq \emptyset} \left(\prod_{j \in S}\e_j \right)\left(\prod_{j' \in S'} \e_{j'}\right) x^3M(x;f\mbf{1},\mbf{L}_{S,S'}) \ll x^3\mc{R}_f(x).
\end{equation*}
The claim of the lemma follows.
\end{proof}
\begin{proof}[Proof of Proposition \ref{VARCALCNP}]
Expanding the square on the left side of \eqref{VARNP} yields
\begin{align*}
&x^{-2} \sum_{d \leq x} |\{n \leq x : f(n+jd) = \e_j , 0 \leq j \leq 2\}|^2 \\
&- \frac{1}{4}\left(x^{-1} \sum_{d \leq x} |\{n \leq x : f(n+jd) = \e_j , 0 \leq j \leq 2\}|\right) + \frac{1}{64}x.
\end{align*}
Combining Lemmata \ref{MOMENT1NP} and \ref{MOMENT2NP} quickly establishes the proposition.
\end{proof}
\section{Sign Patterns of Pretentious Functions in Almost All 4-term APs}
\subsection{Preliminaries and the First Moment Estimate}
Our first quest is to understand the $p$-adic and character local factors $M_p$ and $\Xi_{\mbf{a}}$. In preparation for this, we introduce more notation, some of which is recalled from the introduction. \\
Given a set $S \subseteq \{0,1,2,3\}$, let $\mbf{L}_{S}$ be the collection of forms $\{L_j(n,d) := n+jd: j \in S\}$. Also, write $\mbf{1}_{|S|}$ to denote the vector in $\mb{R}^{|S|}$, all of whose components are 1. \\
Given a fixed prime $p$ we associate to each $\lambda \in \mb{F}_p \bk \{0,1\}$ a non-singular elliptic curve $E_{\lambda}$ defined over $\mb{F}_p$ given by the Legendre model $y^2 \equiv x(x-1)(x-\lambda) (p)$. Finally, we will write $\sideset{}{^{\ast}}\sum_{a(q)}$ to indicate that summation is restricted to residue classes $a$ coprime to $q$. Here and throughout this section, $q \geq 5$ is a positive integer coprime to $6$.
\begin{lem} \label{ELLTRANS}
Let $q \geq 2$ and let $\mbf{a} := (a_0,a_1,a_2,a_3)$ be a vector of integers whose radicals divide $q$. Then $\Xi_{\mbf{a}}(\chi \mbf{1}_{\{0,1,2,3\}}, \mbf{L}_{\{0,1,2,3\}})$ vanishes unless $a_1 = a_2 = a_3 = a_0$, in which case we have
\begin{equation*}
\Xi_{\mbf{a}}(\chi \mbf{1}_{\{0,1,2,3\}}, \mbf{L}_{\{0,1,2,3\}}) = \mu(q)\phi(q)\prod_{p||q} \left(p+1-\#E_{3b^2}(\mb{F}_p)\right),
\end{equation*}
where $b$ is the inverse of $2$ modulo $q$.
\end{lem}
\begin{proof}
Since $\chi$ is primitive and has odd conductor, $q$ must be squarefree, and thus $\chi$ factors as a product of Legendre symbols. Given $b_0,b_1,b_2,b_3$ such that $L_j(n,d)/a_j \equiv b_j (q)$ for $0 \leq j \leq 3$, we observe by the Chinese Remainder Theorem that
\begin{equation*}
\Xi_{\mbf{a}}(\mbf{\chi},\mbf{L}) = \prod_{p|q} \sum_{b_0,b_1,b_2,b_3 (p) \atop \exists n,d : (n+jd)/a_j \equiv b_j(p)} \left(\frac{b_0b_1b_2b_3}{p}\right).
\end{equation*}
We consider several cases depending on the integers $a_j$. \\
\emph{Case 1:} If $p| (a_i,a_j)$ for some $0 \leq i < j \leq 3$ but $p\nmid a_k$ for $k \neq i,j$ then as $p \neq 2,3$, $p|(n,d)$. As such, $b_k \equiv 0(p)$. Hence, the sum over $b_k$ is trivial. Thus, if $p$ divides some two $a_j$'s, $\Xi_{\mbf{a}}$ is trivial unless $a_0 = a_1 = a_2 = a_3$. \\
\emph{Case 2:} Suppose that $p|a_i$ but $p\nmid a_j$ for each $j \neq i$. Then $n+id = pm$ for some $m \in \mb{N}$, and $n+jd \equiv (j-i)d (p)$ for each $j \neq i$. Since $(p,a_j) = 1$, it follows that $(n+jd)/a_j \equiv b_j (p)$ if, and only if, $n+jd \equiv b_ja_j (p)$. Hence, as $a_i$ is squarefree, $(a_i/p,p) = 1$ and
\begin{equation*}
\sum_{b_0,b_1,b_2,b_3 (p) \atop \exists n,d : (n+jd)/a_j \equiv b_j(p)} \left(\frac{b_0b_1b_2b_3}{p}\right) = \sum_{d(p)}\left(d\frac{\prod_{j \neq i} (j-i)a_j}{p}\right)\sum_{m(p)} \left(\frac{m/(a_i/p)}{p}\right) = 0,
\end{equation*}
\emph{Case 3:} We assume now that $a_0 = a_1 = a_2 = a_3$. 
With the above constraints on the $b_j$, we must have $b_3 \equiv 2b_2-b_1$ and $b_4 \equiv 3b_2-2b_1$. Thus, multiplying by $\chi(2b_2)^4\chi(-1)^2 = 1$,
\begin{align*}
\Xi_{\mbf{a}}(\chi \mbf{1}_{\{0,1,2,3\}},\mbf{L}_{\{0,1,2,3\}}) &= \sideset{}{^{\ast}}\sum_{b_1,b_2(q)} \chi(b_1b_2(2b_2-b_1)(3b_2-2b_1)) \\
&= \sideset{}{^{\ast}}\sum_{b_1,b_2(q)} \chi(b_1\bar{b_2}(b_1\bar{b_2}-2)(2b_1\bar{b_2} - 3)) \\
&= \sideset{}{^{\ast}}\sum_{b_1,b_2(q)} \chi(\bar{2}b_1\bar{b_2}(\bar{2}b_1\bar{b_2}-1)(\bar{2}b_1\bar{b_2} - 3\bar{2}^2)).
\end{align*}
Making the change of variables $\bar{2}b_1 \bar{b_2}$ in place of $b_1$, we get
%
\begin{equation} \label{INTERMEDIATE}
\Xi_{\mbf{a}}(\chi \mbf{1}_{\{0,1,2,3\}},\mbf{L}_{\{0,1,2,3\}}) = \sideset{}{^{\ast}}\sum_{d,b_2 (q)}\chi(d(d-1)(d-3b^2)) = \phi(q)\sideset{}{^{\ast}}\sum_{d (q)} \chi(d(d-1)(d-\lambda)).
\end{equation}
Applying the CRT again, the complete character sum factors as
\begin{equation*}
\sideset{}{^{\ast}}\sum_{d (q)}\chi(d(d-1)(d-\lambda)) = \prod_{p|q} \sum_{d (p)} \left(\frac{d(d-1)(d-\lambda)}{p}\right).
\end{equation*}
On the other hand, we know that $1+\sum_{d(p)}\left(1+\left(\frac{d(d-1)(d-\lambda)}{p}\right)\right)$ is precisely the number of $\mb{F}_p$-rational points on $E_{\lambda}$ (including the point at infinity). As such,
\begin{equation}\label{NEAT}
\sideset{}{^{\ast}}\sum_{d (q)}\chi(d(d-1)(d-\lambda)) = \mu(q)\prod_{p|q}\left(p+1-\#E_{\lambda}(\mb{F}_p)\right).
\end{equation} 
Inserting this into \eqref{INTERMEDIATE} proves the claim.
\end{proof}
\begin{lem} \label{TRIV23}
Let $f : \mb{N} \ra \{-1,1\}$ be pretentious to a real character $\chi$ with conductor $q$. 
If $S \subset \{0,1,2,3\}$ has size $2$ or $3$, then $\Xi_{\mbf{a}}(\chi \mbf{1}_{|S|},\mbf{L}_S)= 0$ for any length $|S|$ vector $\mbf{a}$ of divisors of $q$. 
\end{lem}
\begin{proof}
When $|S| = 2$, note that the forms $L_j$ and $L_{j'}$ are linearly independent, and thus any pair of residue classes $(b_j,b_{j'})$ can satisfy the simultaneous congruences $L_j(n,d)/a_j \equiv b_j(q)$ and $L_{j'}(n,d)/a_{j'} \equiv b_{j'}(q)$. By orthogonality, 
\begin{equation*}
\Xi_{\mbf{a}}(\chi \mbf{1}_{|S|},\mbf{L}_S) = \left(\sideset{}{^{\ast}}\sum_{c(q)} \chi(c)\right)^2 = 0.
\end{equation*}
Now suppose that $|S| = 3$, and let $0 \leq j_1 < j_2 < j_3 \leq 3$ be the elements of $S$. A reduction argument similar to (and, in fact, simpler than) the one in Lemma \ref{ELLTRANS} allows one to assume that $a_0 = a_1 = a_2=: a$. Then, $L_{j_t}(n,d)/a \equiv b_t(q)$ for each $1 \leq t \leq 3$ implies that $(j_2-j_1)d \equiv b_2-b_1$, and as $j_2-j_1 \in \{1,2\}$ and $q$ is odd, we have 
\begin{equation*}
b_3 \equiv b_1 + (j_3-j_1)\bar{(j_2-j_1)}(b_2-b_1) (q) =: b_1(1-J)+ Jb_2 (q),
\end{equation*}
where $J := (j_3-j_1)\bar{(j_2-j_1)}$. Note that since $j_3 \neq j_2$ and $j_3-j_1 \in \{2,3\}$, $J \neq 1$ and $J$ and $J-1$ are both invertible. As such, we have
\begin{align*}
\Xi(\chi \mbf{1}_{|S|},\mbf{L}_S) &= \sideset{}{^{\ast}}\sum_{b_1,b_2(q)} \chi(b_1b_2(b_1(1-J)+Jb_2)) \\
&= \chi(J) \sideset{}{^{\ast}}\sum_{b_1,b_2(q)} \chi(b_2)^3\sideset{}{^{\ast}}\sum_{b_1(q)} \chi(b_1\bar{b_2}(1-b_1\bar{b_2}\bar{J}(J-1))).
\end{align*}
Multiplying by $\bar{J}(J-1)$ and making the change of variables $c:= b_1\bar{b_2}\bar{J}(J-1)$ in place of $b_1$ (which is a bijection onto $(\mb{Z}/q\mb{Z})^{\ast}$) yields
\begin{equation*}
\Xi(\chi \mbf{1}_{|S|},\mbf{L}_S) = \chi(J-1)\left(\sideset{}{^{\ast}}\sum_{b_2(q)} \chi(b_2)\right) \left(\sideset{}{^{\ast}}\sum_{c(q)}\chi(c(1-c))\right) = 0.
\end{equation*}
\end{proof}
For the remainder of the paper, we will write $[3] := \{0,1,2,3\}$, and
\begin{equation*}
\Delta_p := p+1-\#E_{3b^2}(\mb{F}_q).
\end{equation*}
We can now state our first moment estimate for sign patterns of pretentious multiplicative functions in almost all 4-term arithmetic progressions.
\begin{prop} \label{FIRSTMOM}
Let $\mbf{\e} \in \{-1,1\}^4$. Let $\delta > 0$ be fixed and let $2 \leq (\log x)^{\delta} \leq z \leq x$ with $z = o(x)$ as $x \ra \infty$. Let $f: \mb{N} \ra \{-1,1\}$ be pretentious to a real character $\chi$ with conductor $q$ coprime to 6. Then
\begin{align} \label{NOQ}
&\frac{1}{z}\sum_{d \leq z\atop q\nmid d} \left|\{n \leq x : f(n+jd) = \e_j \forall j\}\right| = \left(1-\frac{1}{q}\right)\frac{x}{16}+ o(x),
\end{align}
and 
\begin{align}\label{WITHQ}
&\frac{1}{z}\sum_{d \leq z} \left|\{n \leq x : f(n+jd) = \e_j \forall j\}\right| = \frac{x}{16}\left(1+\e_0\e_1\e_2\e_3 \prod_{p\nmid q} M_p(f\chi\mbf{1}_4,\mbf{L}_{[3]})\prod_{p|q}\frac{\mu(p)\Delta_p}{p+1}\right) + o(x).
\end{align}
\end{prop}
\begin{proof}
Applying Lemma \ref{TRIVIAL}, we have
\begin{equation*}
(xz)^{-1}\sum_{d \leq z} \left|\{n \leq x : f(n+jd) = \e_j \forall j\}\right| = \frac{1}{16}\left(1 + \sum_{S \subseteq \{0,1,2,3\} \atop S \neq \emptyset}\left(\prod_{0 \leq j \leq 3} \e_j\right) \sum_{d \leq z} \sum_{n \leq x} \prod_{j \in S} f(n+jd)\right).
\end{equation*}
For each non-empty $S \in \{0,1,2,3\}$, Theorem \ref{MultAvg} (with $\mbf{t} = \mbf{0}$) gives 
\begin{equation} \label{TEMP}
(xz)^{-1}\sum_{d \leq z}\sum_{n \leq x} \prod_{j \in S} f(n+jd) = \sum_{\rad(a_j)|q\atop \forall j \in S}\prod_{j \in S} f(a_j)R(\mbf{a},S)\Xi_{\mbf{a}}(\mbf{L}_S,\chi \mbf{1}_S)\prod_{p \leq x}M_p(F \mbf{1}_{|S|},\mbf{L}_S) + o(1),
\end{equation}
where we have set $$R(\mbf{a},S) := \lim_{x \ra \infty} x^{-|S|} \sum_{\mbf{n} \in [x]^{|S|} \atop a_jq|L_j(\mbf{n}) \forall j \in S} 1.$$ 
When $|S| \in \{2,3\}$, Lemma \ref{ELLTRANS} implies that the right side of \eqref{TEMP} is 0. Now, $\mb{D}(f,\chi;\infty) < \infty$ and for $x$ sufficiently large in terms of $q$,
\begin{align*}
\mb{D}(1,\chi;x)^2 &= \sum_{p\leq x} \frac{1-\chi(p)}{p} = \log_2 x + \log\left(\prod_{p \leq x} \left(1-\frac{\chi(p)}{p}\right)\right) + O(1) \\
&= \log_2 x - \log L(1,\chi) + O(1)\gg \log_2 x.
\end{align*}
By the triangle inequality for $\mb{D}$, it follows that $\mb{D}(f,1;x)^2 \gg \log_2 x$, and hence by Wirsing's theorem (for an effective version due to Hall and Tenenbaum, see Theorem 4.14 of \cite{Ten}),
\begin{equation*}
\sum_{n \leq x} f(n+jd) = \sum_{n \leq x} f(n) + o(x) = o(x),
\end{equation*}
for each $0 \leq j \leq 3$. Thus, when $|S| = 1$, the left side of \eqref{TEMP} is $o(1)$. \\
Now, when we do not sum over multiples of $q$, $R(\mbf{a},\{0,1,2,3\}) = 0$. This implies \eqref{NOQ}. Conversely, if multiples of $q$ are included in the sum then $R(\mbf{a},\{0,1,2,3\}) = (aq)^{-2}$. Thus, for \eqref{WITHQ}, the above and Lemma \ref{ELLTRANS} yields
\begin{align*}
&(xz)^{-1}\sum_{d \leq z} \left|\{n \leq x : f(n+jd) = \e_j \forall j\}\right| \\
&= \frac{1}{16}\left(1+\e_0\e_1\e_2\e_3 \prod_{p \leq x \atop p\nmid q} M_p(F\mbf{1}_4,\mbf{L}_{[3]})\frac{\mu(q) \phi(q)}{q^2}\sum_{\text{rad}(a)|q} \frac{f(a)^4}{a^2}\prod_{p|q}\Delta_p \right) + o(1) \\
&= \frac{1}{16}\left(1+\e_0\e_1\e_2\e_3 \prod_{p \leq x \atop p\nmid q} M_p(F\mbf{1}_4,\mbf{L}_{[3]})\frac{\mu(q)\phi(q)}{q^2}\prod_{p|q}\left(1-\frac{1}{p^2}\right)^{-1}\Delta_p\right) + o(1) \\
&= \frac{1}{16}\left(1+\e_0\e_1\e_2\e_3 \prod_{p \leq x \atop p\nmid q} M_p(F\mbf{1}_4,\mbf{L}_{[3]})\frac{\mu(q)}{q}\prod_{p|q}\left(1+\frac{1}{p}\right)^{-1}\Delta_p\right) + o(1),
\end{align*}
which implies the claim.
\end{proof}
\subsection{The Second Moment Estimate}
We now establish a mean-squared deviation estimate for the cardinalities of the sets $$S_{\mbf{\e}}(d) := \{n \leq x : f(n+jd) = \e_j \ \forall \ 0 \leq j \leq 3\},$$ using the first-moment estimate in Proposition \ref{FIRSTMOM}. In the sequel, let
\begin{equation*}
A_{\mbf{\e}}(f;q) := \e_0\e_1\e_2\e_3 \prod_{p\nmid q} M_p(f\chi\mbf{1}_4,\mbf{L}_{\{0,1,2,3\}})\prod_{p|q}\frac{\mu(p)\Delta_p}{p+1},
\end{equation*}
whenever $f$ is pretentious to a real character $\chi$ modulo $q$. \\
First, we show that the product of $p$-adic local factors is independent of the choice of linear forms for which the factors are defined. As before, given subsets $S,T \subseteq \{0,1,2,3\}$ we let $$\mbf{L}_{S,T} := \{(n,n',d) \mapsto n+jd : j \in S\} \cup \{(n,n',d) \mapsto n'+j'd : j' \in T\}.$$
\begin{lem} \label{PINDEP}
Let $f : \mb{N} \ra \{-1,1\}$ be pretentious to a real character $\chi$ with conductor $q$. Let $S,S',T,T' \subset \{0,1,2,3\}$ be subsets of size 2. Then 
$$\prod_{p \nmid q} M_p(f\chi\mbf{1}_{4},\mbf{L}_{S,T}) = \prod_{p\nmid q} M_p(f\chi \mbf{1}_4,\mbf{L}_{S',T'}).$$
Similarly, 
$$\prod_{p\nmid q}M_p(f\chi\mbf{1}_4,\mbf{L}_{S,\{0,1,2,3\}}) = \prod_{p\nmid q}M_p(f\chi\mbf{1}_4,\mbf{L}_{\{0,1,2,3\},S}) = \prod_{p\nmid q}M_p(f\chi\mbf{1}_4,\mbf{L}_{S',\{0,1,2,3\}}).$$
\end{lem}
\begin{proof}
Fix $p \nmid 6q$ for the moment. Write $S = \{j_0,j_1\}$ and $T = \{k_0,k_1\}$. We can express $M_p(f\chi \mbf{1}_4, \mbf{L}_{S,T})$ as
\begin{equation*}
M_p(f\chi \mbf{1}_4,\mbf{L}_{S,T}) = \sum_{\nu_0,\nu_1,\nu_2,\nu_3 \geq 0} \prod_{0 \leq j \leq 3} F(p^{\nu_j}) \left(\lim_{x \ra \infty} x^{-3} \sum_{n,m,d \leq x \atop p^{\nu_t} || (n+j_td), p^{\nu_{2+t}}||(m+k_td)} 1\right).
\end{equation*}
We split the sum over $n,n'$ and $d$ according to the $p$-adic valuation of each of these variables. Given $m \in \mb{N}$, let $\nu_p(N)$ denote the exponent $r$ such $p^r||N$. Given fixed $n,m$ and $d$ counted by the inner sum above, let $\alpha := \nu_p(n)$, $\beta := \nu_p(m)$ and $\gamma := \nu_p(d)$. By the properties of the $p$-adic valuation, and $p\nmid 6$, if $\alpha \neq \gamma$ then $\nu_t = \max\{\alpha,\gamma\}$ for both $t = 0,1$, and similarly, if $\beta \neq \gamma$ then $\nu_{2+t} = \max\{\beta,\gamma\}$ for both $t = 0,1$. The densities thus depend only on the $p$-adic valuations of $n$, $m$ and $d$, and not on the choice of forms. Hence, these terms are independent of the choices of $j_0,j_1,k_0$ and $k_1$. \\
Suppose now that $\alpha = \gamma$, and write $n' = n/p^{\gamma}$ and $d' := d/p^{\gamma}$. Then it follows that $n' \equiv -j_td' \ (p^{\nu_t-\gamma})$, and $p^{\nu_t-\gamma + 1} \nmid (n' + j_td')$. As such, given $d'$, the density of such $n'$ is $p^{-(\nu_t-\gamma)}\left(1-1/p\right)$, irrespective of the specific choice $j_t$. A similar scenario occurs when the roles of $\alpha$ and $\beta$ are switched, and when $\alpha = \beta = \gamma$. This proves that the $p$-adic factors are independent of $S$ and $T$ when $p \nmid 6q$. \\
We now consider $p |6$. Given $p$, we define an equivalence relation $\sim_p$ among pairs of sets $(S,T)$ with the property that $$(S,T) \sim_p (S',T')  \text{ if, and only if, } M_p(F\mbf{1}_4,\mbf{L}_{S,T}) = M_p(F\mbf{1}_4,\mbf{L}_{S',T'}).$$ We shall furthermore say that a form in $\mbf{L}_{S,T}$ has \emph{bad reduction} at $p$ if the degree in any of the variables of the form decreases upon reduction modulo $p$; we say that the form has \emph{good reduction} at $p$ otherwise. \\
We make the following observations: \\
a) if $\mbf{L}_{S,T}$ consists only of forms of good reduction at $p$ then the arguments above still go through. For instance, $M_2(F\mbf{1}_4,\mbf{L}_{\{0,3\},\{0,1\}}) = M_2(F\mbf{1}_4,\mbf{L}_{\{0,1\},\{0,1\}})$, i.e., $$(\{0,3\},\{0,1\}) \sim_2 (\{0,1\},\{0,1\}).$$ In fact, if $S$ encodes the same number of forms of good and bad reduction mod $p$ as $S'$ does then $(S,T) \sim_p (S',T)$; \\
b) if $S_2$ can be constructed as a translation of $S_1$, and $T$ encodes forms of good reduction then $(S_1,T) \sim_p (S_2,T)$ provided that the set of primes at which forms in $S_1$ have bad reduction only differs by one prime from that of $S_2$. Indeed, this follows from Theorem \ref{MultAvg} because when $z = o(x)$ and $z,x \ra \infty$,
\begin{align*}
&(zx^2)^{-1}\sum_{d \leq z} \sum_{n,n' \leq x} \prod_{j \in S_2} f(n+jd) \prod_{k \in T} f(n'+kd) \\
& =(zx^2)^{-1}\sum_{d \leq z} \sum_{n,n' \leq x} \prod_{j \in S_1} f(n+jd) \prod_{k \in T} f(n'+kd) + o(1).
\end{align*}
and moreover
\begin{align*}
&(zx^2)^{-1}\sum_{d \leq z} \sum_{n,n' \leq x} \prod_{j \in S_i} f(n+jd) \prod_{k \in T} f(n'+kd) \\
&= \left(\sum_{a_j|q \forall j} \frac{f(a_j)}{a_j} R(qa_1,qa_2,qa_3,qa_0) \Xi_{\mbf{a}}(\mbf{L}_{S_i,T},\mbf{\chi})\right)\prod_{p\nmid q}M_p(F\mbf{1}_4,\mbf{L}_{S_i,T}) + o(1),
\end{align*}
The character factors only depend on the \emph{gaps} between the elements of $S_j$, which are invariant under translation. The above remarks thus show that all but possibly the $p$-adic factors for $p|6$ are the same, and if, say, only one of the forms in $S_2$ has bad reduction at $p$ and the other has good reduction everywhere then it also follows that $M_p(F\mbf{1}_4,\mbf{L}_{S_1,T}) = M_p(F\mbf{1}_4,\mbf{L}_{S_2,T})$. \\
With these remarks, we can now complete the proof. It suffices to show that $$(S, \{0,1\}) \sim_p (S',\{0,1\})$$ for all $S,S'$ of size two and both $p = 2$ and 3, as then the same arguments repeat (by transitivity) to fixing $S$ and varying $\{0,1\}$ through all sets of size two. Thus, set $T := \{0,1\}$. First, applying b) twice, we have $$(\{0,1\},T) \sim_2 (\{1,2\},T) \sim_2 (\{2,3\},T).$$ Next, applying a), we have that $(\{1,3\},T) \sim_2 (\{0,1\},T)$, and applying b) again gives $(\{0,2\},T) \sim (\{1,3\},T)$. Finally, by a) we again have $(\{0,3\},T) \sim_2 (\{0,1\},T)$. \\
For $p = 3$, the same sort of arguments work. For instance, in this case we have $(\{0,3\},T) \sim_3 (\{1,3\},T)$ and $(\{0,2\},T) \sim_3 (\{0,1\},T)$ by a), and $(\{1,3\},T) \sim_2 (\{0,2\},T)$. This completes the proof in the $(2,2)$ case. \\
The $(2,4)$ and $(4,2)$ cases follow by similar (and simpler) reasoning.
\end{proof}
\begin{lem}\label{BLANK}
Let $q$ be as above and let $b$ be the inverse of $2$ modulo $q$. For any $p|q$,
\begin{equation} \label{APGOAL}
\sum_{d(p)} \left(\frac{d(d+1)(d+2)(d+3)}{p}\right) = -(\Delta_p+1).
\end{equation}
\end{lem}
\begin{proof}
Let $R$ denote the sum on the left side of \eqref{APGOAL}. The term $d = 0$ contributing nothing, we may restrict the sum to coprime residue classes modulo $p$. Pulling out four factors of $d$ and replacing $d$ by $\bar{d}$, we get
\begin{align*}
R &= \sideset{}{^{\ast}}\sum_{d(p)} \left(\frac{(1+d)(1+2d)(1+3d)}{p}\right) = -1 + \sum_{d(p)} \left(\frac{(1+d)(1+2d)(1+3d)}{p}\right) \\ &= -1 + \sum_{c(p)} \left(\frac{c(2c-1)(3c-2)}{p}\right),
\end{align*}
upon reinserting the term $d = 0$ and making the change of variables $c = d+1$. Removing $c = 0$ and replacing $c$ by $\bar{2c}$, we get
\begin{equation*}
R = -1 + \sideset{}{^{\ast}} \sum_{c(p)} \left(\frac{c(c-1)(c-3\bar{2}^2)}{p}\right) = -(p+2-\#E_{\lambda}(\mb{F}_p)),
\end{equation*}
as in Lemma \ref{ELLTRANS}. This completes the proof.
\end{proof}
\begin{proof}[Proof of Theorem \ref{WITHQTHM}]
We will only consider the case that the number of $+$ signs in $\mbf{\e}$ is odd. The general case is similar but involves further computations of the same type as those involved in the cases we consider.\\
Let $\mc{L}$ denote the left side of \eqref{MSQ}. Expanding the square and applying Proposition \ref{FIRSTMOM}, we have
\begin{align*}
\mc{L} &= z^{-1}\sum_{d \leq z} |S_{\mbf{\e}}(d)|^2 - \frac{x}{8}\left(1+A_{\mbf{\e}}(f;q)\right)\left(z^{-1}\sum_{d \leq z} |S_{\mbf{\e}}(d)|\right) + \frac{x^2}{256}\left(1+A_{\mbf{\e}}(f;q)\right)^2 \\
&= z^{-1}\sum_{d \leq z} |S_{\mbf{\e}}(d)|^2 - \frac{x^2}{256}\left(1+2A_{\mbf{\e}}(f;q)+A_{\mbf{\e}}(f;q)^2\right) + o(x^2).
\end{align*}
We seek to evaluate the second moment of $|S_{\mbf{\e}}(d)|^2$ here. Using Lemma \ref{TRIVIAL}, we get
\begin{align*}
&z^{-1}\sum_{d \leq z} |S_{\mbf{\e}}(d)|^2 = \frac{1}{256z} \sum_{d \leq z} \sum_{n,n' \leq x} \prod_{0 \leq j \leq 3} \left(1+\e_j f(n+jd)\right)\left(1+\e_j f(n'+jd)\right) \\
&= \frac{1}{256} \left(x^2 + 2z^{-1}\sum_{d \leq z} \sum_{n \leq x} \prod_{j \in S} \left(1+\e_jf(n+jd)\right)\right) \\
& +\frac{1}{256}\sum_{S,T \subseteq \{0,1,2,3\} \atop |S|,|T| \geq 1} \left(\prod_{j \in S \atop k \in T} \e_j\e_k\right)z^{-1}\sum_{d \leq z} \sum_{n,n' \leq x} \prod_{j \in S} f(n+jd) \prod_{k \in T} f(n'+kd) \\
&= \frac{1}{256}\left((1+2A_{\mbf{\e}}(f;q) + o(1))x^2+ \sum_{S,T \subseteq \{0,1,2,3\} \atop |S|,|T| \geq 1} \left(\prod_{j \in S\atop k \in T} \e_j\e_k\right)z^{-1}\sum_{d \leq z} \sum_{n,n' \leq x} \prod_{j \in S \atop k \in T} f(n+jd) f(n'+kd)\right),
\end{align*}
so that 
\begin{equation}\label{SIMPDEV}
\mc{L} = \frac{1}{256}\left(\sum_{S,T \subseteq \{0,1,2,3\} \atop |S|,|T| \geq 1} \left(\prod_{j \in S \atop k \in T} \e_j\e_k\right) z^{-1}\sum_{d \leq z} \sum_{n,n' \leq x} \prod_{j \in S \atop k \in T} f(n+jd)f(n'+kd) - A_{\mbf{\e}}(f;q)^2x^2\right) + o(x^2).
\end{equation}
As above, Hal\'{a}sz' theorem implies that when $\min\{|S|,|T|\} = 1$, the contribution here is $o(x^2)$. 
Similarly, by Lemma \ref{TRIV23}, when either $|S| = 3$ or $|T| =3$, the resulting contributions are zero because the character local factor vanishes by orthogonality (because $|S|$ or $|T|$ is odd). \\
Hence, it remains to consider the contributions from $(|S|,|T|) \in \{(2,2),(4,2),(2,4),(4,4)\}$. For each of these contributions,  we can reduce to the case in which the vector $\mbf{a} = (a,\ldots,a,a',\ldots,a')$ upon applying Theorem \ref{MultAvg}, where the components $a$ correspond to forms induced by $S$ and the components $a'$ correspond to forms induced by $T$. For such $\mbf{a}$,
\begin{equation}\label{RDENSITY}
R(\mbf{a},(S,T)) = \lim_{x \ra \infty} x^{-3} \sum_{n,n',d \leq x \atop qa|n+jd, qa' | n'+kd \forall j \in S, k \in T} 1 = \frac{1}{q^2aa'[qa,qa']}.
\end{equation}
We consider the $(4,4)$ term first. 
We apply Theorem \ref{MultAvg}, 
rearranging the character sums as before to get
\begin{align}
&\left(\e_0\e_1\e_2\e_3\right)^2z^{-1}\sum_{d \leq z}\sum_{n,n' \leq x} \prod_{0 \leq j \leq 3} f(n+jd)f(n'+jd) \nonumber\\
&= \frac{x^2}{q^2} \prod_{p \nmid q} M_p(f\chi\mbf{1}_8,\mbf{L}_{\{0,1,2,3\},\{0,1,2,3\}})\sum_{\text{rad}(a),\text{rad}(a')|q} \frac{1}{aa'[qa,qa']} \nonumber\\
&\cdot \mathop{\sum_{b_0,b_1(q)}\sum_{c_0,c_1(q)}}_{\exists d : d/a \equiv b_1-b_0(q), d/a' \equiv c_1-c_0(q)} \chi(b_0b_1(2b_1-b_0)(3b_1-2b_0))\chi(c_0c_1(2c_1-c_0)(3c_1-2c_0)) +o(x^2). \label{DOUBLECHARSUM}
\end{align}
A routine (and tedious) argument shows that $$M_p(f\chi \mbf{1}_8, L_{\{0,1,2,3\},\{0,1,2,3\}}) = M_p(f\chi \mbf{1}_4,L_{\{0,1,2,3\}})^2$$ for each $p \nmid q$. Next, consider the character sum in \eqref{DOUBLECHARSUM}. For those $d$ specified by the congruence condition, write $d = [a,a']m$. Then $b_1 \equiv b_0 + ma'/(a,a') \ (q)$ and $c_1 \equiv c_0 + ma/(a,a') \ (q)$. Put $A := a/(a,a')$ and $A' := a'/(a,a')$, noting that they are coprime. We may then rewrite the sum as
\begin{equation*}
\mc{S}_{A,A'}(\chi) := \sum_{b,c,m(q)} \chi(b(b+mA')(b+2mA')(b+3mA'))\chi(c(c+mA)(c+2mA)(c+3mA)).
\end{equation*}
For $a,b \in \mb{Z}/q\mb{Z}$ write $P_a(b) := b(b+a)(b+2a)(b+3a)$. By the CRT, we split $\mc{S}_{A,A'}(\chi)$ as
\begin{equation*}
\mc{S}_{A,A'}(\chi) = \prod_{p|q}\sum_{m,b,c(p)} \left(\frac{P_{mA'}(b)}{p}\right)\left(\frac{P_{mA}(c)}{p}\right).
\end{equation*}
If $p|A'$ then the sum over $b$ is $p-1$, and note by coprimality that $p\nmid A$ so that $A$ is invertible modulo $p$. Replacing $c$ by $c\bar{A}$, we get
\begin{align*}
&(p-1)\sum_{m,c(p)} \left(\frac{P_m(c)}{p}\right) = (p-1)^2 + (p-1)\sideset{}{^{\ast}}\sum_{m(p)} \sum_{c(p)}\left(\frac{P_m(c)}{p}\right) \\
&= (p-1)^2\left(1+ \sum_{c(p)}\left(\frac{P_1(c)}{p}\right)\right) = (p-1)^2\mu(p)\Delta_p,
\end{align*}
upon invoking Lemma \ref{BLANK}. The same term occurs for primes $p$ dividing $A$. Now suppose that $p\nmid AA'$. As before, we can replace $(b,c)$ by $(b\bar{A'},c\bar{A})$. Then, separating $m = 0$ from the remaining terms of the sum once again and then changing variables a second time, we get
\begin{align*}
&(p-1)^2 + \sideset{}{^{\ast}} \sum_{m(p)} \sum_{b,c(p)} \left(\frac{P_m(b)}{p}\right)\left(\frac{P_m(c)}{p}\right) = (p-1)^2 + (p-1)\sum_{b,c(p)} \left(\frac{P_1(b)}{p}\right)\left(\frac{P_1(c)}{p}\right) \\
&= (p-1)^2 + (p-1)\left(\sum_{b(p)}\left(\frac{P_1(b)}{p}\right)\right)^2 = (p-1)^2 + (p-1)(\Delta_p+1)^2.
\end{align*}
As such, we have
\begin{equation} \label{SAAX}
\mc{S}_{A,A'}(\chi) = \phi(q)^2\prod_{p|AA'}\mu(p)\Delta_p \prod_{p|q \atop p\nmid AA'} \left(1+\frac{(\Delta_p+1)^2}{p-1}\right).
\end{equation}
%
%
Returning to \eqref{DOUBLECHARSUM}, we next evaluate the expression
\begin{align}
&\frac{\phi(q)^2}{q^2}\prod_{p|q}\left(1+\frac{(\Delta_p+1)^2}{p-1}\right) \sum_{\text{rad}(a), \text{rad}(a')|q} \frac{1}{aa'[qa,qa']} \prod_{p|aa'/(a,a')^2} \frac{(p-1)\mu(p)\Delta_p}{p-1+(\Delta_p+1)^2}\nonumber\\
&= \frac{\phi(q)^2}{q^3}\prod_{p|q} \left(1+\frac{(\Delta_p+1)^2}{p-1}\right) \sum_{\text{rad}(\delta)|q} \frac{1}{\delta^3} \sum_{\text{rad}(a), \text{rad}(a')|q \atop (a,a') = 1} \frac{1}{(aa')^2} \prod_{p|aa'} \frac{(p-1)\mu(p)\Delta_p}{p-1+(\Delta_p+1)^2}.\label{INTERIM}
\end{align}
The inner sum can be written as the product
\begin{equation*}
\prod_{p|q \atop p \nmid a} \left(1-\frac{(p-1)\Delta_p}{p-1+(\Delta_p+1)^2}\sum_{k \geq 1} p^{-2k}\right) = \prod_{p|q \atop p \nmid a} \left(1-\frac{\Delta_p}{(p+1)(p-1+(\Delta_p+1)^2)}\right),
\end{equation*}
so that the sum over $a$ becomes
\begin{align*}
&\prod_{p|q}\left(1-\frac{\Delta_p}{(p+1)(p-1+(\Delta_p+1)^2)}\right) \prod_{p|q}\left(1-\frac{(p^2-1)\Delta_p}{(p+1)(p-1+(\Delta_p+1)^2)-\Delta_p}\sum_{k \geq 1} p^{-2k}\right)  \\
&= \prod_{p|q}\left(1-\frac{\Delta_p}{(p+1)(p-1+(\Delta_p+1)^2)}\right)\left(1-\frac{\Delta_p}{(p+1)(p-1+(\Delta_p+1)^2)-\Delta_p}\right) \\
&= \prod_{p|q} \frac{(p+1)(p-1+(\Delta_p+1)^2)-2\Delta_p}{(p+1)(p-1+(\Delta_p+1)^2)}
\end{align*}
Inserting this expression into \eqref{INTERIM} yields
\begin{align*}
&\frac{\phi(q)}{q^3}\prod_{p|q}\left(1-p^{-3}\right)^{-1}\frac{(p+1)(p-1+(\Delta_p+1)^2)-2\Delta_p}{p+1} \\
&= \prod_{p|q} \frac{p^2-1 + (p+1)(\Delta_p^2+2\Delta_p + 1)-2\Delta_p}{(p+1)\left(p^2+p + 1\right)} \\
&= \prod_{p|q} \frac{\Delta_p^2 + p^2 + p(\Delta_p+1)^2}{(p+1)((p+1)^2-p)} \\
&= \prod_{p|q} \left(\frac{\mu(p)\Delta_p}{p+1}\right)^2\frac{(1+1/\Delta_p)^2 + 1/p + p/\Delta_p^2}{1+1/p(p+1)}.
\end{align*}
It remains to determine the contributions from the $(2,2)$, $(4,2)$ and $(2,4)$ terms. We will only consider the $(2,2)$ case, the $(4,2)$ (and symmetrically) the $(2,4)$ case being similar and simpler.\\
The $(2,2)$ contribution is 
\begin{equation} \label{22SUMS}
\sum_{S,T \in \{0,1,2,3\} \atop |S| = |T| = 2} \left(\prod_{j \in S} \e_j\right)\left(\prod_{k \in T} \e_k\right) \left(z^{-1}\sum_{d \leq z} \sum_{n,n' \leq x} \prod_{j \in S} f(n+jd) \prod_{k \in T} f(n'+ kd) \right).
\end{equation}
Fix a pair of sets $S$ and $T$. We apply Theorem \ref{MultAvg} to get
\begin{align*}
&(zx^2)^{-1}\sum_{d \leq z}\sum_{n,n' \leq x} \prod_{j \in S} f(n+jd) \prod_{k \in T} f(n'+ kd) \\
&= \sum_{\rad(a_j)|q \forall j} f(a_j) R(\mbf{a},(S,T)) \sum_{b_0,b_{1},c_0,c_{1}(q) \atop \exists n,n',d : (n+j_td)/a_t \equiv b_t (q), (n'+k_td)/a_{2+t} \equiv b_{2+t} (q)} \chi\left(b_0b_{1}c_0c_{1}\right) \prod_{p \nmid q} M_p(f\chi \mbf{1}_4, \mbf{L}_{S,T}).
\end{align*}
By Lemma \ref{PINDEP}, the product of $p$-adic factors $M_p$ is independent of $S$ and $T$, and we will show that the same is true of the character sum. Indeed, arguing as in our treatment of the $(4,4)$ case, the sums over $b_j$ are non-zero only when $a = a_0 = a_{1}$ and $a' = a_2 = a_{3}$. Writing $d = [a,a']m$ as before, we can rewrite the character sum as
\begin{align}\label{22CHARSUM}
\sum_{b,c(q)} \chi(b(b+mA'))\chi(c(c+mA)).
\end{align}
We apply the CRT and consider the character sum modulo each prime $p$ dividing $q$ as above. When $p|A$ we get
\begin{align}
(p-1)\sum_{b,m(p)}\left(\frac{b(b+m)}{p}\right) &= (p-1)\left((p-1) + \left(\frac{-1}{p}\right)\sideset{}{^{\ast}} \sum_{m(p)} \sum_{b(p)} \left(\frac{b(m-b)}{p}\right)\right) \nonumber\\
&= (p-1)^2\left(1+\left(\frac{-1}{p}\right)J\left(\left(\frac{\cdot}{p}\right),\left(\frac{\cdot}{p}\right)\right)\right), \label{JACOBI}
\end{align}
the last term being a Jacobi sum. It is well-known that this Jacobi sum is precisely $-\left(\frac{-1}{p}\right)$, so the right side of \eqref{JACOBI} vanishes. The same is true for $p|A'$, and hence the non-zero contributions come from $A = A' = 1$, i.e., from $a = a'$. As such, for each prime $p|q$,
\begin{align*}
\sum_{b,c,m(p)} \left(\frac{b(b+m)}{p}\right)\left(\frac{c(c+m)}{p}\right) &= (p-1)^2 + \sideset{}{^{\ast}} \sum_{m(p)}\sum_{b,c(p)} \left(\frac{b(b+m)}{p}\right)\left(\frac{c(c+m)}{p}\right) \\
&= (p-1)^2 + (p-1)\left(\sum_{b(p)} \left(\frac{b(b+1)}{p}\right)\right)^2 \\
&= (p-1)^2 + (p-1)\left(\frac{-1}{p}\right)^2J\left(\left(\frac{\cdot}{p}\right),\left(\frac{\cdot}{p}\right)\right)^2 \\
&= p(p-1).
\end{align*}
This expression is then clearly independent of $S$ and $T$. Thus, summing over $S$ and $T$, \eqref{22SUMS} becomes
\begin{align*}
&\frac{1}{q^3}\prod_{p\nmid q} M_p(f\chi \mbf{1}_4, \mbf{L}_{\{0,1\},\{0,1\}}) \left(\sum_{\rad(a)|q} \frac{1}{a^3}\right)\prod_{p|q}p(p-1) \sum_{S,T \subset \{0,1,2,3\} \atop |S| = |T| = 2} \sum_{j \in S, k \in T} \e_j\e_k \\
&= \prod_{p|q}\frac{p}{p^2+p+1} \prod_{p\nmid q} M_p(f\chi \mbf{1}_4, \mbf{L}_{\{0,1\},\{0,1\}}) \left(\sum_{0 \leq i < j \leq 3} \e_i \e_j\right)^2.
\end{align*}
Now, when $\mbf{\e}$ has an odd number of $+$ signs, $\sum_{0 \leq i \leq 3} \e_i = \pm 2$ and thus we have
\begin{equation*}
\sum_{0 \leq i < j \leq 3} \e_i\e_j = \frac{1}{2}\left(\left(\sum_{0 \leq i \leq 3} \e_i\right)^2 - \sum_{0 \leq i \leq 3} \e_i^2\right) = 0.
\end{equation*}
This implies that the $(2,2)$ contribution is $o(x^2)$. This factor is also responsible for making the $(4,2)$ and $(2,4)$ contributions vanish.
Thus, in the end, we have
\begin{align*}
\mc{L} &=  \frac{x^2}{256} \prod_{p \nmid q} M_p(f\chi\mbf{1}_4,\mbf{L}_{\{0,1,2,3\}})^2\\
&\cdot \prod_{p|q} \left(\frac{\mu(p)\Delta_p}{p+1}\right)^2\left(\prod_{p|q}\frac{(1+1/\Delta_p)^2 + 1/p + p/\Delta_p^2}{1+1/p(p+1)} - 1\right) + o(x^2) = \frac{x^2}{256}(T_{4,4}-A_{\mbf{\e}}(f;q)^2) + o(x^2)
\end{align*}
which proves Theorem \ref{WITHQTHM}.
\end{proof}
\begin{proof}[Proof of Theorem \ref{AE}]
We follow the proof of Theorem \ref{WITHQTHM}. The summation in $d$ is restricted such that $q \nmid d$ and normalized by $(1-1/q)z$, and $0$ stands in place of $A_{\mbf{\e}}(f;q)$ (as in Proposition \ref{FIRSTMOM}). Note that in \eqref{RDENSITY}, the restriction $q \nmid d$ there implies that the quantity $R(\mbf{a},(S,T)) = 0$ for $|S|,|T| \geq 2$. This means that all of the contributions by $|S|,|T| \geq 2$ are $o(x^2)$ in \eqref{SIMPDEV}. It thus follows that
\begin{equation*}
z^{-1} \sum_{d \leq z \atop q\nmid d} \left(|S_{\mbf{\e}}(d)|-\frac{x}{16}\right)^2 = o(x^2).
\end{equation*}
The conclusion then follows by Chebyshev's inequality.
\end{proof}
\begin{rem}\label{GENERIC}
By Hasse's bound, we always have $|\Delta_p| \leq 2\sqrt{p}$. In fact, it is known \cite{MiM} that as $p \ra \infty$, 
\begin{equation*}
\pi(x)^{-1} |\{p \leq x : \cos^{-1}(\Delta_p/2\sqrt{p}) \in I\}| \ra \mu_{ST}(I) := \begin{cases} \frac{2}{\pi}\int_I \sin^2 u du &\text{ if $E_{\lambda}$ is non-CM} \\ \frac{1}{2}\left(1_{\pi/2 \in I} + |I|\right) &\text{ if $E_{\lambda}$ is CM}.\end{cases}
\end{equation*}
for all intervals $I \subset [-1,1]$. 
Now, if $q$ is fixed and we choose an elliptic curve $E_{\lambda}/\mb{Q}$ such that $\lambda \equiv 3 \bar{2}^2(q)$ then in general (whether or not $E_{\lambda}$ is CM) we do not understand the behaviour of $\Delta_p$ for $p|q$. Instead, we may draw a heuristic from a result of Miller and Murty \cite{MiM}, which states that for a fixed $p$ and a one-parameter family of elliptic curves $\{E_t/\mb{F}_p : t \in \mb{F}_p\}$, the \emph{discrepancy}
\begin{equation*}
\max_{I \subseteq [0,\pi]} \left|p^{-1}|\{t \in \mb{F}_p : \cos^{-1}(\Delta_{t,p}/2\sqrt{p}) \in I\}|-\mu_{ST}(I)\right|
\end{equation*}
tends to 0 as $p \ra \infty$, where $\Delta_{t,p}$ is the trace of Frobenius on $E_t/\mb{F}_p$. This says roughly that for generic curves over $\mb{F}_p$ in a family, the angles $\cos^{-1}(\Delta_{t,p}/2\sqrt{p})$ behave as they should for elliptic curves over $\mb{Q}$ in the $p$-limit.\\
Thus, if we assume that the element $3\bar{2}^2$ modulo $p$ yields a generic element of the one-parameter family generated by the Legendre models $E_t : y^2 = x(x-1)(x-t)$ with $t \in \mb{F}_p$ then for generic $p$, $\Delta_p \asymp \sqrt{p}$ but $\Delta_p$ is \emph{not} close to $\sqrt{p}$. The same is true modulo $q$, for $q$ a product of more than one prime. Thus, on heuristic grounds the product 
\begin{equation*}
\prod_{p|q}\frac{(1+1/\Delta_p)^2 + 1/p + p/\Delta_p^2}{1+1/p(p+1)}
\end{equation*}
is $\asymp 1$, but on the other hand, it is not asymptotically 1 as $p \ra \infty$ in general. 
\end{rem}
\nocite{TaT}
\bibliographystyle{plain}
\bibliography{gowers}

\begin{thebibliography}{10}

\bibitem{ElB}
Y.~Buttkewitz and C.~Elsholtz.
\newblock Patterns and complexity of multiplicative functions.
\newblock {\em J. London. Math. Soc}, 84:578--594, 2011.

\bibitem{CHO}
S.~Chowla.
\newblock {\em The Riemann Hypothesis and Hilbert's Tenth Problem}.
\newblock Gordon and Breach Science Publishers, New York-London-Paris, 1965.

\bibitem{DeB}
N.G. DeBruijn.
\newblock On the number of positive integers $\leq x$ and free of prime factors
  $> y$.
\newblock {\em Nederl. Akad. Wetensch. Proc. Ser. A}, 54:50--60, 1951.

\bibitem{ErdPart}
P.~Erd\H{o}s.
\newblock On an elementary proof of some asymptotic formulae in the theory of
  partitions.
\newblock {\em Ann. of Math.}, 43:437--450, 1942.

\bibitem{HF}
N.~Frantzikinakis and B.~Host.
\newblock Asymptotics for multilinear averages of multiplicative functions.
\newblock {\em Math. Proc. Camb. Phil. Soc.}, 161:87--101, 2016.

\bibitem{HF2}
N.~Frantzikinakis and B.~Host.
\newblock Higher order fourier analysis of multiplicative functions and
  applications.
\newblock {\em J. Amer. Math. Soc.}, 30:67--157, 2017.

\bibitem{GT}
B.~Green and T.~Tao.
\newblock Linear equations in primes.
\newblock {\em Ann. of Math.}, 171:1753--1850, 2010.

\bibitem{Klu}
O.~Klurman.
\newblock Correlations of multiplicative functions and applications.
\newblock {\em Compositio Math.}, 153:1622--1657, 2017.

\bibitem{MRT}
K.~Matom\"{a}ki, M.~Radziwill, and T~Tao.
\newblock An averaged form of {C}howla's conjecture.
\newblock {\em Algebra and Number Theory}, 9:2167--2196, 2015.

\bibitem{MRTSigns}
K.~Matom\"{a}ki, M.~Radziwill, and T.~Tao.
\newblock Sign patterns of the {L}iouville and {M}\"{o}bius functions.
\newblock {\em Forum Math. Sigma}, 4, 2016.
\newblock e14, 44 pp.

\bibitem{Matth}
L.~Matthiesen.
\newblock Linear correlations of multiplicative functions.
\newblock arXiv:1606.04482.

\bibitem{MiM}
S.J. Miller and M.R. Murty.
\newblock Effective equidistribution and {S}ato-{T}ate law for families of
  elliptic curves.
\newblock {\em J. Number Theory}, 131:25--44, 2011.

\bibitem{Sil}
J.~Silverman.
\newblock {\em The Arithmetic of Elliptic Curves, 2nd Ed.}
\newblock Graduate Texts in Mathematics, Springer, New York, NY, 2008.

\bibitem{Tao}
T.~Tao.
\newblock The logarithmically averaged {C}howla and {E}lliott conjectures for
  two-point correlations.
\newblock arXiv:1509.05422 [math.NT].

\bibitem{HOFA}
T.~Tao.
\newblock {\em Higher Order Fourier Analysis}.
\newblock Graduate Studies in Mathematics, American Mathematical Society,
  Providence, RI, 2012.

\bibitem{TaT}
T.~Tao and J.~Ter\"{a}v\"{a}inen.
\newblock The structure of logarithmically averaged correlations of
  multiplicative functions, with applications to the {C}howla and {E}lliott
  conjectures.
\newblock arXiv:1708.02610.

\bibitem{Ten}
G.~Tenenbaum.
\newblock {\em Introduction to Analytic and Probabilistic Number Theory}.
\newblock Cambridge University Press, Cambridge, UK, 1994.

\bibitem{TOT}
L.~T\'{o}th.
\newblock Multiplicative arithmetic functions of several variables: a survey.
\newblock {\em in: T.M. Rassias, P.M. Pardalos (Eds.) Mathematics without
  Boundaries, in: Surveys in Pure Mathematics, Springer, New York},
  84:483--514, 2014.

\end{thebibliography}
\end{document}